\documentclass[12pt]{amsart}

\usepackage{amssymb}

\newtheorem{thm}{Theorem}[section]
\newtheorem{cor}[thm]{Corollary}

\newtheorem{lem}[thm]{Lemma}
\newtheorem{prop}[thm]{Proposition}

\theoremstyle{definition}
\newtheorem{defi}[thm]{Definition}

\newcommand{\ts}{\vspace{0.5\baselineskip}\noindent{\bf Proof.}$\;\;$}
\newcommand{\pfs}{{\unskip\nobreak\hfill\hbox{ $\Box$}\vspace{\baselineskip}}}

\newcommand{\ZZ}{{\bf Z}}
\newcommand{\QQ}{{\bf Q}}
\newcommand{\RR}{{\bf R}}

\newcommand{\CC}{{\bf C}}

\newcommand{\PP}{{\bf P}}

\newcommand{\mL}{{\mathbb L}}

\newcommand{\bSS}{{\mathbb S}}
\newcommand{\cA}{{\mathcal A}}

\newcommand{\cC}{{\mathcal C}}
\newcommand{\cD}{{\mathcal D}}
\newcommand{\cE}{{\mathcal E}}

\newcommand{\cG}{{\mathcal G}}
\newcommand{\cS}{{\mathcal S}}
\newcommand{\cH}{{\mathcal H}}

\newcommand{\cK}{{\mathcal K}}
\newcommand{\cL}{{\mathcal L}}
\newcommand{\cM}{{\mathcal M}}

\newcommand{\cO}{{\mathcal O}}
\newcommand{\cP}{{\mathcal P}}
\newcommand{\cQ}{{\mathcal Q}}
\newcommand{\cU}{{\mathcal U}}
\newcommand{\cV}{{\mathcal V}}
\newcommand{\cW}{{\mathcal W}}
\newcommand{\cT}{{\mathcal T}}

\newcommand{\HK}{{hyper-K\"ahler}}
\newcommand{\tH}{{\tilde{H}}}
\newcommand{\lra}{\longrightarrow}

\DeclareMathOperator{\Pic}{Pic}

\DeclareMathOperator{\Aut}{Aut}
\DeclareMathOperator{\Hom}{Hom}

\DeclareMathOperator{\Hilb}{Hilb}
\DeclareMathOperator{\Br}{Br}
\DeclareMathOperator{\rk}{rk}
\DeclareMathOperator{\sym}{Sym}

\begin{document}

\title{Contractions of hyper-K\"ahler fourfolds and the Brauer group}
\author[B.~van Geemen]{Bert van Geemen}
\address{Dipartmento di Matematica,Universit\`a di Milano, Via Saldini 50, 20133 Milano, Italy.}
\email{lambertus.vangeemen (at) unimi.it}

\author[G.~Kapustka]{Grzegorz Kapustka}

\address{Department of Mathematics and Informatics,
Jagiellonian University, {\L}ojasiewicza 6, 30-348 Krak{\'o}w, Poland.}
\email{grzegorz.kapustka (at) uj.edu.pl}
\thanks{
G.K.\ is supported by the project Narodowe Centrum Nauki 2018/30/E/ST1/00530,
B.v.G.\ was supported by the PRIN 2015 project Geometry of Algebraic Varieties}

\begin{abstract}
We study the geometry of exceptional loci of birational contractions of \HK\ fourfolds
that are of K3$^{[2]}$-type.
These loci are conic bundles over K3 surfaces and we determine their classes in the Brauer group.
For this we use the results on twisted sheaves on K3 surfaces,
on contractions  and on the corresponding Heegner divisors.
For a general K3 surface of fixed degree
there are three (T-equivalence) classes of order two Brauer group elements. 
The elements in exactly two of these classes are represented by conic bundles on such fourfolds. 
We also discuss various examples of such conic bundles.
\end{abstract}
 
\maketitle
\tableofcontents

\section*{Introduction}
A \HK\ variety is a simply-connected smooth projective variety admitting a unique
(up to scalar) non-degenerate holomorphic two-form.
A \HK\ surface is a K3 surface and the divisorial contractions are well-known in this case.
If $E\subset S$ is a prime divisor on a K3 surface $S$ with $E^2<0$,
then $E$ is a smooth rational curve and $E^2=-2$. Moreover, $E$ can be contracted to an ODP point
and any divisorial contraction with prime exceptional divisor is obtained in this way.
Our aim is to discuss an analogous statement for projective \HK\ fourfolds $X$
of K3$^{[2]}$ type. These fourfolds are obtained by deforming the blow-up along the diagonal
of the symmetric square of a K3 surface. The Beauville-Bogomolov-Fujiki
quadratic form on the second cohomology group of $X$ is denoted by $q$.

Let $X$ be a projective fourfold of K3$^{[2]}$ type (appropriately general)
containing a prime divisor $E\subset X$ with $q(E)<0$. Then there is
a primitive big and nef divisor (class) $H$ on $X$ 
such that the map induced by $H$ contracts $E$ to a K3 surface $K$ (possibly after some flops).

$$
\begin{array}{rcl} \phi_{H}:\,X&\longrightarrow&Y\\\cup&&\cup\\
 E&\longrightarrow&K~.
\end{array}
$$

The description of the nef cone of $X$, due to  Hassett-Tschinkel \cite{HTnef}, Markman \cite{Markman_surv}, Bayer–Macr\`\i\ \cite{BM}, Bayer–Hassett–Tschinkel \cite{BayerHT},
and Mongardi \cite{Mongardi}, implies that
there are then two possibilities for $q(E)$,
it is either $-2$ or $-8$; the first
case are the BN (Brill-Noether) contractions whereas the latter case are the
HC (Hilbert-Chow) contractions. In both cases the contraction 
induces a map $p:E\rightarrow K$ which 
is a $\PP^1$-fibration over a K3 surface $K$.
In the HC case this $\PP^1$-fibration is the projectivisation 
of the tangent bundle of $K$, $E=\PP(\cT_K)\rightarrow K$.

In the BN case, the $\PP^1$-fibration $E\rightarrow K$ 
can have a non-trivial Brauer class in $\Br(K)_2$,
the two-torsion subgroup of the Brauer group of $K$.
Then $E\not\cong\PP(\cV)$ for a rank two vector bundle $\cV$ on $K$,
equivalently, there is no rational section of the map
$E\rightarrow K$. 
Using the work of Debarre and Macr\`\i\ \cite{DM}, which is based on a study of the sublattice
of $H^2(X,\ZZ)$ generated by $H$ and $E$,
we find that there are four types of BN contractions. 
These correspond to irreducible components of certain
Heegner divisors in period spaces of polarized \HK\ fourfolds of K3$^{[2]}$ type. 
In two cases one has $E=\PP(\cV)$ where $\cV$ is the Mukai bundle, a rigid stable vector bundle
of rank two on $K$, see Corollary \ref{ko}.
These two cases differ only in the degree of the polarization on $K$.
In the other two cases the fibrations $E\rightarrow K$ have a non-trivial Brauer class.

To determine the Brauer classes, we first show that $X$ is a moduli
space of $\beta$-twisted sheaves on some K3 surface $T$ for some Brauer class $\beta\in \Br(T)_2$.
From the results of Bayer and Macr\`\i\ in \cite{BM} 
we then deduce that $T=K$ and that $\beta$ is also
the Brauer class defined by the $\PP^1$-fibration $E\rightarrow K$, see Theorem \ref{BNBr}.
 
Given a polarized K3 surface $(K,h)$ with $\Pic(K)=\ZZ h$ and $h^2=2d$,
not all two-torsion elements in $\Br(K)$ are obtained as the class of 
a $\PP^1$-fibration $E\rightarrow K$ for a BN exceptional divisor $E$. 
We classify the non-trivial Brauer classes $\alpha\in \Br(K)_2$ as follows (see Theorem \ref{vG}). 
For any B-field $B\in H^2(K,\QQ)$ which represents $\alpha$, 
the intersection number (modulo the integers) $Bh\in\{0,1/2\}$ is an invariant of $\alpha$.
In case $4Bh+h^2\equiv 0\mod 4$ there is a further invariant of $\alpha$
which is $B^2\in\{0,1/2\}$, again modulo the integers.
In case $4Bh+h^2\equiv 2\mod 4$, the value of $B^2$ is not an invariant of $\alpha$ 
but it is convenient to choose $B$ such that $B^2=1/2$.
There are thus three types of order two Brauer classes on a generic polarized K3 surface. 
The type corresponds to the isometry class
of the index two sublattice of the transcendental lattice of $K$ determined by the 
Brauer class.
This refines the classification given in \cite{vG}.

The non-trivial Brauer classes determined by BN exceptional divisors are those with
$B^2=1/2$. 
Using the notation for Heegner divisors which is
explained in Section \ref{heegner}, Theorem \ref{BNBr} and Corollary \ref{ko} imply:

\begin{thm}\label{mainthm}
Let $(K,h)$ be a polarized K3 surface with $\Pic(K)=\ZZ h$ and $h^2=2d$
and let  $\alpha\in \Br(K)_2$ be a non-trivial two-torsion Brauer class.

Then $\alpha$ has a B-field representative with $B^2=1/2$ 
if and only if
there is a conic bundle $E\rightarrow K$ with Brauer class $\alpha$
which is the contraction of an exceptional divisor on a \HK\ fourfold
$X$ of K3$^{[2]}$ type induced by a big and nef divisor class $H$.

The period point of $(X,H)$ lies in $\cD^{(1)}_{2d,8d,\alpha}$ if 
$Bh=1/2$ and in $\cD^{(1)}_{8d,8d,\beta}$ if $Bh=0$.
\end{thm}

It is easy to show, see Proposition \ref{degreeE}, 
that the exceptional divisor $E$ of a BN contraction has canonical divisor $K_E$ 
with self-intersection number $K_E^3=12$.
Any $\PP^1$-fibration $p:E\rightarrow K$ over a K3 surface $K$
is isomorphic to $E=\PP(\cU)$ for an $\alpha$-twisted locally free rank two sheaf $\cU$ over $K$,
where $\alpha\in \Br(K)_2$ is the Brauer class determined by $E$. 
Using basic properties of twisted sheaves, we found that there exists a non-trivial conic bundle 
$p:E\rightarrow K$ with
$K_E^3=12$ only in case $E$ determines a Brauer class $\alpha$ having a B-field representative
with $B^2=1/2$. 

So the basic properties of twisted sheaves already imply that at most two of the three
types of Brauer classes are obtained from BN exceptional divisors on \HK\ fourfolds of
$K3^{[2]}$ type.

\

A brief outline of the paper is as follows. Basic results on $\PP^1$-fibrations
are recalled in Section \ref{conic}.
There we also discuss examples and methods to compute Chern classes associated to such 
fibrations. 
In Section \ref{bcts} we give the classification of Brauer classes and the basic
properties of twisted sheaves. 
The next section introduces the Heegner divisors and describes the \HK\ fourfolds
admitting contractions. These descriptions are used in Section \ref{BrBNdiv} to prove
the main result Theorem \ref{mainthm}. We also discuss relations between the second
cohomology groups of $X$, $E$ and $K$. In Section \ref{exaHeegner}
we study  examples of conic bundles over K3 surfaces in \HK\ fourfolds and their relations to 
classical constructions. We characterise in this way the Brauer classes of known conic bundles.
In the last section we attempt a description of certain conic bundles over quartic K3 surfaces
but we obtain only partial results. 
By brute force computations we find 
that the Hilbert scheme of two points on the Fermat quartic surface defines
an interesting EPW sextic $Y\subset \PP^5$. In fact the
singular locus of $Y$, which is a smooth surface in general,
is now itself singular at $60$ points.

\section*{Acknowledgments}

We thank O.~Debarre, E.~Macr\`\i\ and P.\ Stellari for helpful discussions.

\section{Conic bundles}\label{conic}

\subsection{The invariants of a conic bundle}\label{chernP}
Let 
$$
p:\,E\,\longrightarrow\,K
$$
be a conic bundle over a K3 surface $K$. We assume that $E$ is everywhere non-degenerate, 
so all fibres are smooth rational curves. We will also refer to $E$ 
as a $\PP^1$-fibration over $K$. 
We compute the basic invariants associated to $E$,
following \cite{Sar}, \cite{CI}.

The relative dualizing sheaf $\omega_{E/K}$, a line bundle on $E$, 
is defined by the exact sequence
$$
0\,\longrightarrow\,p^*\Omega^1_K\,\longrightarrow\, \Omega^1_E\,\longrightarrow \,
\omega_{E/K}\,\longrightarrow\,0~.\qquad (\ast)
$$

Taking determinants we find, since $\omega_K=\cO_K$, that
$$
\omega_E\,\cong\,p^*\omega_K\otimes \omega_{E/K}\,\cong\,\omega_{E/K}~,
$$
so the relative dualizing sheaf is just the canonical bundle.

The Chern classes of $E$ can be computed from the exact sequence above:
$$
c_1(E)\,=\,-K_E,\quad c_2(E)\,=\, p^*c_2(K)\,=24f,\quad 
c_3(E)\,=\,-c_3(\Omega^1_E)=-p^*c_2(K)\omega_E\,=\,48~,
$$
where $f$ be the class of a fiber of $p$ and thus $\omega_Ef=\omega_{E/K}f=-2$ where we used
that $(\omega_{E/K})_{|f}=\omega_f$ (\cite[Prop.\ II.8.10]{HAG}).

\begin{prop} \label{propcw}
For a conic bundle $p:E\rightarrow K$ as above let $\cW:=p_*\omega_{E/K}^{-1}$.
Then $\cW$ is a locally free sheaf of rank three on the K3 surface $K$ and
$$
c_1(\cW)\,=\,0,\qquad K_E^2\,=\, -c_2(\cW)f,\qquad K^3_E\,=\,2c_2(\cW)~.
$$
\end{prop}
\ts
We apply Grothendieck-Riemann-Roch (GRR) to the sheaves $\omega_{E/K}^{\pm 1}$ and the map 
$p:E\rightarrow K$.

As $(\omega_{E/K})_{|f}=\omega_f$ which has degree $-2$ on $f\cong\PP^1$,
we have $R^0p_*\omega_{E/K}=0$ and $R^1p_*\omega_{E/K}$ is a line bundle on
$K$. Using \cite[Exc.\ 3.12, p.85]{HM} we find that $R^1p_*\omega_{E/K}=\cO_K$. Thus
$p_!\omega_{E/K}=-[\cO_K]$ and $\mbox{ch}(p_!\omega_{E/K})=-1_K$.
Since $\mbox{td}(\cT_K)=1+2P$, where $P\in K$ is a point, and
GRR states that $\mbox{ch}(p_!\omega_{E/K})\mbox{td}(\cT_K)=
p_*(\mbox{ch}(\omega_{E/K})\mbox{td}(\cT_E))$, we find 
$$
-(1+2P)\,=\,p_*(\mbox{ch}(\omega_{E/K})\mbox{td}(\cT_E))~.
$$
Using that $\omega_{E/K}=\omega_E$ is a line bundle with $c_1(\omega_{E/K})=K_E$
and the results from above we find
{\renewcommand{\arraystretch}{1.5}
$$
\begin{array}{rcl}
\mbox{ch}(\omega_{E/K})&=&1+K_E+\frac{1}{2}K_E^2+\frac{1}{6}K_E^3~,\\
\mbox{td}(\cT_E)&=&1\,-\,\mbox{$\frac{1}{2}$}K_E\,+\,\mbox{$\frac{1}{12}$}(K_E^2+24f)\,-\,fK_E
~.
\end{array}
$$
}
Hence GRR gives
$$
-1\,-\,2P\,=\,
p_*(1\,+\,\mbox{$\frac{1}{2}$}K_E\,+\,\mbox{$\frac{1}{12}$}(K_E^2+24f)\,+\,fK_E)~.
$$
Since $\dim E=3>\dim K=2$ we get $p_*1=p_*1_E=0$. Similarly $p_*f=0$ 
since the one dimensional fiber is contracted to a point in $K$. Next $p_*K_E=(-2)\cdot 1_KX$
since $K_Ef=\deg(K_f)=-2$.
Thus the degree zero terms agree. In degree one we must have 
$0=p_*(\mbox{$\frac{1}{12}$}K_E^2)$, hence 
$$
K_E^2\,=\,c_Ef
$$ 
for some integer $c_E$. 

Next we apply GRR to $\omega_{E/K}^{-1}$.
From $(\omega_{E/K})_{|f}=\omega_f$ we also get $(\omega_{E/K}^{-1})_{|f}\cong\cO_{\PP^1}(2)$, 
in particular $h^0(f,(\omega_{E/K}^{-1})_{|f})=3$ for all fibers of $p$ so 
$\cW=p_*\omega_{E/K}^{-1}$ is locally
free of rank three on $K$. 
Moreover,
$H^1(f,(\omega_{E/K}^{-1})_{|f})=0$ for all fibers of $p$ and thus $R^ip_*(\omega_{E/K}^{-1})=0$
for $i>0$. This implies that $p_{!}\omega_{E/K}^{-1}=p_*\omega_{E/K}^{-1}$ and since
$\omega_{E/K}=\cO_E(-K_E)$, GRR gives
$\mbox{ch}(\cW)\mbox{td}(\cT_K)\,=\,
p_*(\mbox{ch}(\cO_E(-K_E))\mbox{td}(\cT_E))$.
We find:
{\renewcommand{\arraystretch}{1.5}
$$
\begin{array}{rcl}
\mbox{ch}(\cW)&=&3+c_1(\cW)+\frac{1}{2}(c_1(\cW)^2-c_2(\cW))~,\\
\mbox{td}(\cT_K)&=&1+\frac{1}{12}c_2(K)\;=\;1+2P~,\\
\mbox{ch}(\cO(-K_E))&=&1-K_E+\frac{1}{2}K_E^2-\frac{1}{6}K_E^3~,\\
\mbox{td}(\cT_E)&=&1-\mbox{$\frac{1}{2}$}K_E+\mbox{$\frac{1}{12}$}(K_E^2+24f)-fK_E~.
\end{array}
$$
}
The left hand side of GRR is thus:
$$
\mbox{ch}(\cW)\mbox{td}(\cT_K)\,=\,3+c_1(\cW)+(6P+\mbox{$\frac{1}{2}$}(c_1(\cW)^2-c_2(\cW)))~,
$$
whereas the right hand side is:
$$
p_*(\mbox{ch}(\cO_E(-K_E))\mbox{td}(\cT_E))\,=\,p_*\left(
1\,-\,\mbox{$\frac{3}{2}$}K_E\,+\,(\mbox{$\frac{13}{12}$}K_E^2+2f)\,-\,
( \mbox{$\frac{1}{2}$}K_E^3 \,+\,3fK_E)
\right)~.
$$
We already found that $K_E^2=c_Ef$ for some constant $c_E$, hence $p_*K^2_E=p_*f=0$. 
Thus the degree one term on the right hand side is zero and therefore also $c_1(\cW)=0$.
Finally we have $K_E^3=c_EfK_E=-2c_E$ and, identifying a zero cycle with its degree,
$p_*fK_E=fK_E=-2$,
so the last terms on the right hand side are 
$$
p_*(\mbox{$\frac{1}{2}$}K_E^3 \,+\,3fK_E)\,=\,-c_E\,-\,6~.
$$
Comparing with the LHS and using $c_1(\cW)=0$ we get:
$$
6\,-\,c_2(\cW)\,=\,c_E\,+\,6~,\qquad\mbox{hence}\quad
c_2(\cW)\,=\,-c_E~.
$$
Substituting this in $K_E^2=c_Ef$ and $K_E^3=-2c_E$ one finds the last two equalities. \pfs

\subsection{The Brauer class of a conic bundle}\label{gluingB}
Let $E\to K$ be a $\PP^1$-fibration over a K3 surface. Since $\Aut(\PP^1)=PGL_2(\CC)$,
the glueing data for $E$ provides a Cech cocycle which defines
a class $[E]\in H^1(K,PGL_2(\cO_K))$.
From the exact sequence 
$$
0\lra \cO_K^{\ast}\lra GL_2(\cO_K)\lra PGL_2(\cO_K)\lra 0
$$
we obtain a coboundary map $ H^1(PGL_2(\cO_K))\to H^2(\cO_K^{\ast})$.
The image of $[E]$ through this map is a two-torsion element that we denote by 
$\alpha_E\in H^2(\cO_K^{\ast})_{tors}=\Br(K)$, the Brauer group of $K$.
We denote by $\Br(K)_2$ the subgroup of two-torsion elements in $\Br(K)$.

In case $\alpha_E=0$, the $\PP^1$-fibration $E$ is obtained from a class 
in $ H^1(K,GL_2(\cO_K))$. Such a class defines a rank two vector bundle ${\mathcal V}$ 
and $E\cong\PP({\mathcal V})$ is then called a trivial $\PP^1$-fibration or a trivial conic bundle.

Since $E_K\cdot f=-2$, one has $h^0(E,nK_E)=0$ for all $n>0$. We now consider $h^0(E,-K_E)$.

\begin{prop}\label{h0-KE} Let $p:E\rightarrow K$ be a non-trivial $\PP^1$-fibration over a K3 surface 
$K$ with $\rk\, \Pic(K)=1$. Then $h^0(E,-K_E)=0$.
\end{prop}
\ts
Assume that $s\in H^0(-K_E)$, $s\neq 0$, and let $D\in |-K_E|$ be the 
effective divisor on $E$ defined by $s$. Then there is the exact sequence:
$$
0\,\lra\,\cO_E(K_E)\,\stackrel{s}{\lra}\,\cO_E\,\lra\,\cO_D\,\lra\,0~.
$$
Recall that $h^0(K_E)=0$, moreover 
$h^1(K_E)=h^{3,1}(E)=h^{2,0}(E)=1$ and $h^1(\cO_E)=h^{0,1}(E)=0$. Hence $h^0(D,\cO_D)=2$.

Since $Df=+2$, at least one (and at most two) of the irreducible components of $D$ map onto $K$. 
If there are two such components or if the unique such component is not reduced, 
then $p$ has a rational section and hence $E=\PP(\cV)$ for a vector bundle $\cV$ on $K$
(cf.\ \cite[Lemma 3.5]{Shramov}) which contradicts the non-triviality $E$.

So we now assume that $D$ has a unique reduced, irreducible,
`horizontal' component $D_h$ mapping onto $K$. 
The support of any other component of $D$ would be the inverse image of a curve in $K$. 
Since $\rk\,\Pic(K)=1$ the divisor $D$ is then linearly equivalent to $D':=D_h+p^*C$ where
$C$ is a (reduced, irreducible) curve in $K$. 
But then $D'\in|-K_E|$ is connected and reduced, contradicting $h^0(\cO_{D'})=2$. 
We conclude that $h^0(E,-K_E)=0$.
\qed

\subsection{The vector bundle $\cW$}
Let $\omega_{E/K}$ be the relative dualizing sheaf of the conic bundle $p:E\rightarrow K$
as in Section \ref{chernP}, it is a line bundle on $E$. 
As in Proposition \ref{propcw} we define
$$
\cW\,:=\,p_*\omega_{E/K}^{-1}~,
$$
then $\cW$ is a locally free sheaf of rank three on the K3 surface $K$ with $c_1(\cW)=0$.
The following proposition is well-known.

\begin{prop}\label{reldual}
The sheaf $\cW$ is a locally free sheaf of rank three on $K$ which is self dual: 
$\cW\cong\cW^\vee$.
There is a quadratic form $Q:\cW\rightarrow \cW^\vee\cong\cW$
such that $E$ is the associated conic bundle of $(\cW,Q)$.

Conversely, given a locally free rank three sheaf $\cE$ on $K$ with a nowhere degenerate 
quadratic form $Q$, let $p:E\subset\PP(\cE)\rightarrow K$ be the conic bundle defined by $Q$.
Then $p_*\omega^{-1}_{E/K}\cong \cE$.
\end{prop}

\ts 
See \cite{Sar}, \cite{CI}.
\pfs

\subsection{Chern classes for conic bundles on a double plane}\label{exac2}
Let $\pi:K\rightarrow\PP^2$ be a double cover branched along a smooth sextic curve $C$ 
defined by $f_6=0$. There are three types of non-trivial conic
bundles on $K$, they are associated to a point of order two, an odd or an even theta characteristic
on $C$ (see \cite{vG}, \cite{IOOV} and Theorem \ref{vG}).

We recall well-known conic bundles $E\rightarrow K$ which have Brauer classes 
of each of these types and we (re)compute the selfintersection $K_E^3$ of their canonical bundle
in two cases, see Proposition \ref{c2order2}, Corollary \ref{c2cubpl} respectively.
We will discuss some implications of these results at the end of \S \ref{tsaa}.

\

\begin{prop} \label{c2order2}
Let $\pi:K\rightarrow\PP^2$ be a general K3 surface of degree two with branch curve $C$.
A point of order two in $\Pic(C)$ defines a conic bundle $p:E\rightarrow K$ 
with $K_E^3=12$ and $c_2(p_*\omega^{-1}_{E/K})=6$.
\end{prop}

\ts
In view of Proposition \ref{propcw}, it suffices to show that $c_2(p_*\omega^{-1}_{E/K})=6$. 
Let $L\in \Pic(C)$ be of order two, so $L^\vee=L$, where $L^\vee$ is the dual line bundle. 
Let $j:C\hookrightarrow \PP^2$ be the inclusion map.
There is a resolution (\cite{Ca}, the case $e=3$ in 
\cite[Proposition 4.6]{beauville_det})
$$
0\,\lra\,\cO_{\PP^2}(-4)^3\,\stackrel{M}{\lra}\,\cO_{\PP^2}(-2)^3\,\lra\,j_*L\,\lra\,0
$$
with a symmetric matrix $M$. The coefficients $M_{ij}$ of the $3\times 3$-matrix $M$ are 
elements of $\mbox{Hom}(\cO_{\PP^2}(-4),\cO_{\PP^2}(-2))=H^0(\cO_{\PP^2}(2))$. 
The determinant of $M$ is an equation of the sextic curve $C$.

We will use \cite[Section 2]{DK} to define a conic bundle $E\rightarrow K$ as a subvariety of
$\PP( p_*\omega_{E/K}^{-1})$.
Twisting the sequence by $\cO_{\PP^2}(3)$ and pulling it back along the finite 
(hence affine and flat) map $\pi$ to the K3 surface $K$ 
we get the exact sequence, with $\cO_K(d):=\pi^*\cO_{\PP^2}(d)$:
$$
0\,\lra\,\cE\,\stackrel{\pi^*M}{\lra}\,\cE^\vee\,\lra\,\pi^*j_*L(3)\,\lra\,0,
\qquad \cE\,:=\,\cO_K(-1)^3~.
$$
The determinant of this morphism of vector bundles is $\pi^*\det(M)\in H^0(\cO_K(6))$
which has divisor $D=\pi^*C=2C_K$, where $i:C\hookrightarrow C_K\subset K$ is the 
ramification curve, which is isomorphic to $C$. 

The kernel $\cK$ of $\pi^*M$ is the sheaf on $D$ defined as 
$$
\cK\,:=\;\ker(\pi^*M_{|D}\,:\,\cE_{|D}\,\longrightarrow\,\cE_{|D}^\vee)~,\qquad 
D=2C_K=\pi^*C~.
$$
We need to know the restriction of $\cK$ to the curve $C_K$. Since
$\pi:C_K\rightarrow C$ is an isomorphism 
we have an exact sequence:
$$
0\,\lra\,\cK_{|C_K}\,\lra\,\cE_{|C_K}\,\stackrel{M_{|C_K}}\lra\,\cE^\vee_{|C_K}\,\lra\,L(3)\,
\lra\,0~,
$$
where now $L$ is viewed as a line bundle on $C_K$.
Since $M_{|C_K}$ is symmetric, after dualizing the surjection of vector bundles on $C$ above,
we get $L(-3)\hookrightarrow \cE_{|C}\rightarrow\cE^\vee_{|C}$ and thus $\cK_{|C_K}=L(-3)$.

Define a vector bundle $\tilde{\cE}$ on $K$ 
as the dual of the kernel of the surjection 
$$
\tilde{\cE}^\vee\,:=\,\ker(\cE^\vee\,\lra\, \cE_{X|C_K}^\vee\,\lra\,
\cK_{|C_K}^\vee=i_*L(3)\,\lra\,0)~,
$$
so that we have the exact sequence (see \cite[(7)]{DK})
$$
0\,\lra\,\tilde{\cE}^\vee\,\lra\,\cE^\vee\,\lra\,i_*L(3)\,\lra\,0~.
$$
The quadratic form $\pi^*M$ on $\cE^\vee$
defines a nowhere degenerate quadratic form on $\tilde{\cE}$.
The conic bundle $E\rightarrow K$ is defined by this quadratic form on $\tilde{\cE}$.

By Proposition \ref{reldual} we have that $p_*\omega_{E/K}^{-1}=\tilde{\cE}$.
To compute the Chern classes of $\tilde{\cE}$ we 
recall that, with $h:=c_1(\cO_K(1))$, the total Chern class of $\cE^\vee$ is
$$
c(\cE^\vee)=(1+h)^3=1+3h+3h^2\quad\mbox{so}\quad 
c_1(\cE^\vee)\,=\,3h,\quad c_2(\cE^\vee)=3h^2=6~.
$$
Computing the Chern classes of $i_*L(3)$ with \cite[\S 2.1, Lemma 1]{Friedman} we get:
$$
c_1({\cE}^\vee)=c_1(\tilde{\cE}^\vee)+[C_K],\qquad\mbox{hence} \quad
c_1(\tilde{\cE}^\vee)\,=\,c_1({\cE}^\vee)-[C_K]\,=\,0~,
$$
as expected since $\tilde{\cE}$ should be self-dual, and 
$$
c_2(\cE^\vee)=c_2(\tilde{\cE}^\vee)+c_1(\tilde{\cE}^\vee)[C_K]+([C_K]^2-i_*c_1(L(3))),
$$
as $c_1(\tilde{\cE}^\vee)=0$ and $[C_K]^2-i_*c_1(L(3))=(3h)^2-3h\cdot3h=0$, we find
$c_2(\tilde{\cE})=c_2(\tilde{\cE}^\vee)=6$. Using $c_2(\cE)=c_2(\cE^\vee)$ we are done.
\pfs

\subsection{Computing Chern classes of Azumaya algebras}
The following proposition allows us to determine the Chern classes of
some Azumaya algebras on a K3 double plane. The idea is based on \cite{IK} and \cite{Ku35},
but Proposition \ref{chA} simplifies the computations. 

\begin{prop} \label{chA}
Let $\pi:K\rightarrow \PP^2$ be a K3 double cover of $\PP^2$
and let $\cA$ be a coherent sheaf on $K$ of rank $r$ with $c_1(\cA)=0$. 
Then we have $c_1(\pi_*\cA)\,=\,-3rH$, with $H:=c_1(\cO_{\PP^2}(1))$, and:
$$
c_2(\cA)\,=\,-\mbox{$\frac{9}{2}$}r(r-1)\,+\,c_2(\pi_*(\cA))~.
$$
\end{prop}

\ts
Recall that the GRR for a map $\pi:T\rightarrow S$ and a coherent sheaf $\cA$ of rank $r$
on $T$ states:
$$
{\rm ch (\pi_!\cA)\mbox{td}(\cT_S})\,=\,\pi_*(\rm ch (\cA)\mbox{td}(\cT_T))~.
$$
In case $T$ is a K3 surface (so $K_T=-c_1(\cT)=0$ and $c_2(\cT)=24P$, with $P$ a point in $T$) 
and $S=\PP^2$ (so $K_S=-3H$ and $c_2(\cT_S)=3H^2$). 
For a double cover $\pi:T\rightarrow S$ 
we have $f_!\cA=f_*\cA$ since $\pi$ is affine and thus $R^i\pi_*\cA=0$ for $i>0$.
The left hand side of GRR is then 
{\renewcommand{\arraystretch}{1.3}
$$
\begin{array}{rcl}
&&\mbox{ch} (\pi_*\cA)\mbox{td}(\cT_{\PP^2})\\
&=&2r\,+\,(3rH\,+\,c_1(\pi_*(\cA)))\,+\,(2r\,+\,\mbox{$\frac{3}{2}$}c_1(\pi_*(\cA))H\,+\,
\mbox{$\frac{1}{2}$}c_1(\pi_*\cA)^2\,-\,c_2(\pi_*(\cA)))~.
\end{array}
$$
}
The right hand side is
{\renewcommand{\arraystretch}{1.3}
$$
\begin{array}{rcl}
\pi_*(\mbox{ch} (\cA)\mbox{td}(\cT_T))
&=&2r\,+\,\pi_*c_1(\cA)\,+\,(2r\,+\,
\mbox{$\frac{1}{2}$}\pi_*(c_1(\cA))^2\,-\,c_2(\cA))~,
\end{array}
$$
}
where we view the codimension two cycles as integers.

Since $c_1(\cA)=0$ one finds
$
c_1(\pi_*\cA)\,=\,-3rH
$ and
$$
c_2(\cA)\,=\,-\left(\mbox{$\frac{3}{2}$}c_1(\pi_*(\cA))H\,+\,
\mbox{$\frac{1}{2}$}c_1(\pi_*\cA)^2\,-\,c_2(\pi_*(\cA)) \right)~.
$$
Substituting for $c_1(\pi_*\cA)$ and using $H^2=1$, one finds the formula for $c_2(\cA)$.
\pfs

The following corollary computes $K_E^3$ for the conic bundle over the double plane $K$
defined by an odd theta characteristic on the branch curve. These conic bundles are also
obtained from cubic fourfolds with a plane, see \S \ref{fanoplane}. 

\begin{cor}\label{c2cubpl}
Let $\pi:K\rightarrow\PP^2$ be a general K3 surface of degree two with branch curve $C$.
An odd theta characteristic on $C$ defines a conic bundle $p:E\rightarrow K$ with
$K_E^3=12$ and $c_2(p_*\omega^{-1}_{E/K})=6$.
\end{cor}

\ts
Let $L$ be an odd theta characteristic on $C$, so $L\otimes L\cong \omega_C$ and $h^0(L)$ is odd.
Since $K$ is general, $h^0(L)=1$ and $j_*L$, where $j:C\hookrightarrow\PP^2$ 
has a resolution (see \cite[Prop.\ 4.2(b)]{beauville_det})
$$
0\,\lra\,\cO_{\PP^2}(-2)^3\oplus\cO_{\PP^2}(-3)\,\stackrel{M}{\lra}\,
\cO_{\PP^2}(-1)^3\oplus\cO_{\PP^2}\,\stackrel{}{\lra}\,j_*L\,\lra\,0~,
$$ 
where $M$ is a symmetric matrix with $\det(M)=f_6$, a degree six polynomial defining $C$.
Thus $M$ also gives a quadratic form
$$
M:\,\cE:=\,\cO^3\,\oplus\,\cO(-1)\,\lra\,\cO(1),\qquad
M\,=\,\left(\begin{array}{cccc} a_{11}&a_{21}&a_{31}&b_1\\
             a_{21}&a_{22}&a_{23}&b_2\\
             a_{31}&a_{32}&a_{33}&b_3\\
             b_1&b_2&b_3&c
            \end{array}\right)~,
$$
with $a_{ij},b_k,c$ global sections of $\cO(1)$, $\cO(2)$ and $\cO(3)$ respectively.
The rulings of the quadric surface defined by $M_x$ are the fibers of the conic bundle $E$
over the points in $\pi^{-1}(x)$ for $x\in\PP^2$.

Following \cite[p.306-308]{IK}, \cite[\S 3]{Ku35},
the (sheaf of) even Clifford algebra(s) of $M$ is the sheaf 
$$
Cl_0(M)\,=\,\left(\oplus_{k= 0}^2 \cE^{\otimes 2k}\otimes \cO_{\PP^2}(-k)\right)/I~,
$$
where $I$ is generated by the local sections $v\otimes v\otimes\lambda-q(v)\lambda$ 
with $v,\lambda$ local sections of $\cE$ and $\cO_{\PP^2}(-1)$ respectively, so that $q(v)\lambda$
is a local section of $\cO$. The rank of $Cl_0(q)$ as a sheaf of $\cO_S$-modules is $2^{4-1}=8$.  
Using the fact that the center of $Cl_0(M)$ is the rank two sheaf $\pi_*\cO_K$, one can show that 
$\pi_*\cA=Cl_0(M)$, where $\cA$ is the rank four Azumaya algebra on $K$ corresponding to
the conic bundle $E$ defined by the odd theta characteristic.
One then finds that:
$$
\pi_*\cA\,=\,\cO_{\PP^2}\,\oplus\,\cO_{\PP^2}(-1)^{\oplus 3}\,\oplus\,
\cO_{\PP^2}(-2)^{\oplus 3}\,\oplus\,\cO_{\PP^2}(-3)~.
$$
The total Chern class is thus, with $H=c_1(\cO_{\PP^2}(1))$:
$$
c(\pi_*\cA)\,=\,1\cdot (1-H)^3(1-2H)^3(1-3H)\,=\,1\,-\,12H\,+60H^2~,
$$
hence $c_1(\pi_*\cA)=-12H$ and $c_2(\pi_*\cA)=60$.
From Proposition \ref{chA} we get:
$$
c_2(\cA)\,=\,-\mbox{$\frac{9}{2}$}r(r-1)\,+\,c_2(\pi_*(\cA))\,=\,-9\cdot 2\cdot 3\,+\,60=\,6~.
$$
Using the trace map, $\cA=\cO\oplus\cA_0$
and $\cA_0\cong p_*\omega_{E/K}^{-1}$ \cite[Corollary 6.5]{CI}, we get
$c_2(p_*\omega_{E/K}^{-1})=6$ and thus $K_E^3=12$ by Proposition\ \ref{propcw}.
\pfs

\

\section{Brauer classes and twisted sheaves}\label{bcts}

\subsection{Brauer classes and B-fields}\label{BBL}
We introduced the two-torsion subgroup $\Br(K)_2$ of the Brauer group $H^2(\cO_K^*)_{tors}$
of a K3 surface $K$ in \S \ref{gluingB}.
The exponential exact sequence gives the exact sequence
$$
0\longrightarrow H^2(K,\ZZ)/\Pic(K)\longrightarrow H^2(K,\cO_K)
\longrightarrow H^2(\cO_K^*)\longrightarrow 0
$$
where we identify $\Pic(K)=H^1(K,\cO_K^*)$ with its image in $H^2(K,\ZZ)$.
Any $\alpha\in \Br(K)_2$ has a lift $\tilde{\alpha}$ to the complex vector space $H^2(\cO_K)$ and
since $2\tilde{\alpha}\in H^2(K,\ZZ)/\Pic(K)$ there is a
$B=B_\alpha\in \mbox{$\frac{1}{2}$}H^2(K,\ZZ)$ mapping to
$\tilde{\alpha}$. We say that $B$ is a B-field representative of $\alpha$.
If $B,B'$ map to the same $\alpha$ then $B=B'+\mbox{$\frac{1}{2}$}p+c$ for some 
$p\in \Pic(K)$ and some $c\in H^2(K,\ZZ)$ (\cite[\S 4]{HuybrechtsTwist}, \cite[\S 6]{Kuz}). 

We now assume that $\Pic(K)=\ZZ h$ and $h^2=2d>0$.
Since $h^2\in 2\ZZ$, for any $a\in\ZZ$ and $c\in  H^2(K,\ZZ)$ 
the intersection number (in $\mbox{$\frac{1}{2}$}\ZZ$)
$$
Bh\,=\,(B'+\mbox{$\frac{a}{2}$}h+c)h\,\equiv B'h\,\mod\,\ZZ
$$
is an invariant of $\alpha$. Since $2B\in H^2(K,\ZZ)$ we have $Bh\in 
\mbox{$\frac{1}{2}$}\ZZ/\ZZ=\{0,\mbox{$\frac{1}{2}$}\}$.

Kuznetsov gave a geometrical interpretation of this invariant
in terms of the restriction of a conic bundle representing $\alpha$ to curves in $K$
in \cite[Lemma 6.2]{Kuz}. 

Generalizing \cite[Lemma 6.1]{Kuz} to any $d$, we find another invariant of $\alpha$.
Since $\alpha$ is two-torsion, 
the B-field $B$ can be written as $B=B_0/2$ with $B_0\in H^2(K,\ZZ)$.
The lattice $H^2(K,\ZZ)$ is even, so $B^2=B_0^2/4\in \mbox{$\frac{1}{2}$}\ZZ$.

\begin{lem} \label{parity}
Let $K$ be a K3 surface with $\Pic(K) = \ZZ h$ and $h^2 = 2d$.
Let $\alpha\in \Br(K)_2$ and let $B\in \mbox{$\frac{1}{2}$}H^2(K,\ZZ)\subset H^2(K,\QQ)$ 
be a B-field representing $\alpha$.
If $4Bh+h^2\equiv 0\mod 4$, then also $B^2\mod \ZZ$ is an invariant of $\alpha$.
\end{lem}

\ts
With the notation as above, $c^2,\,2B'c,\,2\mbox{$\frac{a}{2}$}ch=\,ach\in\ZZ$ and thus
$$
B^2\,=\,(B'+\mbox{$\frac{a}{2}$}h+c)^2
\,\equiv\,(B')^2+\mbox{$\frac{a^2}{4}$}h^2+aB'h\mod\ZZ~.
$$
Since $Bh=B'h\mod \ZZ$, we also have $ 4B'h+h^2\equiv 0\mod 4$,
and thus $\mbox{$\frac{a^2}{4}$}h^2+aB'h\in \ZZ$
since $h^2$ is even and $a^2\equiv a\mod 2$.
\pfs

For a polarized K3 surface $(K,h)$ with $\Pic(K)=\ZZ h$ and $h^2=2d$, a Brauer class 
$\alpha\in \Br(K)_2$, with B-field representative $B$, thus has the invariant 
$Bh\in \mbox{$\frac{1}{2}$}\ZZ/\ZZ$ and:

\noindent
if $d$ is odd and $Bh\not\in\ZZ$ there is a further
invariant $B^2\in \mbox{$\frac{1}{2}$}\ZZ/\ZZ$,

\noindent
if $d$ is even and $Bh\in\ZZ$ there is a further invariant 
$B^2\in \mbox{$\frac{1}{2}$}\ZZ/\ZZ$.

For any $d$ we have thus distinguished three types of Brauer classes of order two.

\

The next lemma guarantees that a class in $\Br(K)_2$
can be represented by a B-field which has good properties.

\begin{lem} \label{reprB}
Let $K$ be a K3 surface with $\Pic(K) = \ZZ h$ and $h^2 = 2d$.
Let $\alpha\in \Br(K)_2$, then there is a B-field 
$B\in \mbox{$\frac{1}{2}$}H^2(K,\ZZ)\subset H^2(K,\QQ)$ 
representing $\alpha$
such that  $2B\in H^2(X,\ZZ)$ is primitive and such that 
$B^2,Bh\in\{0,\mbox{$\frac{1}{2}$}\}\in\QQ$.
\end{lem}

\ts
In case $\alpha=0\in \Br(K)_2$, we take $B=0$. So we may assume that $\alpha\neq 0$. 
Choose a B-field $B\in\mbox{$\frac{1}{2}$}H^2(K,\ZZ)$ representing $\alpha$.

Since $B\not\in H^2(K,\ZZ)+\mbox{$\frac{1}{2}$}\Pic(K)$,
the sublattice $M:=<2B,h>$ of $H^2(K,\ZZ)$ has rank two.
If $M$ is degenerate, that is $(2B)^2h^2-(2Bh)^2\neq 0$,
$Bh$ must be an integer since $(2B)^2$ and $h^2$ are even integers. Hence $Bh=m$ for some $m\in \ZZ$.
Since $h\in H^2(K,\ZZ)$ is primitive, there is a $t\in H^2(K,\ZZ)$ with $ht=1$ and thus
$(B-mt)h=0$. Since $B,B-mt$ define the same Brauer class, in case
$Bh\equiv 0\mod \ZZ$ we may assume that $Bh=0$.
Then $M$ is non-degenerate except if also $B^2=0$. In that case, write $B=mB'$ for a primitive
$B'\in \mbox{$\frac{1}{2}$}H^2(K,\ZZ)$ and an integer $m$. Since $\alpha\neq 0$, $m$ is odd 
and $B'$ is primitive, also defines $\alpha$, and $B'h=0$, $(B')^2=0$ so we replace $B$ by $B'$.

In the other cases, $M$ is non-degenerate and $M\subset M_{pr}\subset H^2(K,\ZZ)$, 
where $M_{pr}$ is the primitive closure of $M$, which
are all the $v\in H^2(K,\ZZ)$ for which an integer multiple lies in $M$. 
Then $M_{pr}$ is a nondegenerate primitive sublattice of $H^2(K,\ZZ)$. 
By results of Nikulin \cite[Theorem 1.14.4]{N}, the embedding of $M_{pr}$ in $H^2(K,\ZZ)$ is unique up to isometry.
Moreover, there is an isomorphism $H^2(K,\ZZ)\cong U^3\oplus E_8(-1)^2$ 
such that $M_{pr}$ maps to $U^2$.
Since $h\in M_{pr}$ is primitive, we may assume it maps to $(1,d)_1$ (in the first copy of $U$)
where $h^2=2d$. We will write $B=B_1+B_2\in \mbox{$\frac{1}{2}$}(U\oplus U)$.

In case $Bh=0$, one finds that $B_1$ must be  
an integer multiple of $\mbox{$\frac{1}{2}$}(1,-d)=(0,-d)+\mbox{$\frac{1}{2}$}(1,d)$.
Hence $B_1\in H^2(K,\ZZ)+\mbox{$\frac{1}{2}$}\Pic(K)$ and we
may assume that $B=B_2=\mbox{$\frac{1}{2}$}(a,b)_2$ for some $a,b\in\ZZ$. 
Adding a suitable element of $U_2\subset  H^2(K,\ZZ)$ we may assume that $a,b\in\{0,1\}$.
Then $2B$ is primitive in $H^2(K,\ZZ)$, $B^2\in\{0,\mbox{$\frac{1}{2}$}\}$ and we are done.

In case $Bh\equiv \mbox{$\frac{1}{2}$}\mod \ZZ$, write 
$B=\mbox{$\frac{1}{2}$}(a,b)_1+(a',b')_2$ and we may assume that $a,b,a',b'\in\{0,1\}$. 
As $(B-a\mbox{$\frac{1}{2}$}h)_1=\mbox{$\frac{1}{2}$}(0,c)_1$, we may also assume that $a=0$ and 
then $b=1$ (since $B'h\not\equiv 0$). For such a $B$ we do have $Bh=\mbox{$\frac{1}{2}$}$
and $B^2=\mbox{$\frac{1}{2}$}a'b'\in \{0,\mbox{$\frac{1}{2}$}\}$, moreover $2B'$ is primitive
in $H^2(K,\ZZ)$.
\qed

\

\subsection{Brauer classes and lattices}
There is an isomorphism (cf.\ \cite{vG})
$$
\Br(K)\,:=\,H^2(\cO_K^{*})_{tors}\,=\,\Big(H^2(K,\ZZ)/c_1(\Pic(K))\Big)\otimes_\ZZ(\QQ/\ZZ)\,=\,
\mbox{Hom}_\ZZ(T_K,\QQ/\ZZ),
$$
where $T_K:=c_1(\Pic(K))^\perp\subset H^2(K,\ZZ)$ is the transcendental lattice of $K$.
The kernel of a non-trivial $\alpha\in \Br(K)_2$ 
is a sublattice of index two in $T_K$ denoted by $\Gamma_\alpha$:
$$
\Gamma_{\alpha}\,=\, \ker(\alpha:\,T_K\,\longrightarrow\,\ZZ/2\ZZ)~.
$$

In case $\Pic(K)=\ZZ h$, with $h^2=2d$,
we classify elements of order two in the Brauer group by the isomorphism class of 
the lattice $\Gamma_\alpha$ (this is called T-equivalence in \cite{HS_EquivTwisted}).
We find the same three types as before, see Theorem \ref{vG}.

To get a more explicit description of the $\Gamma_\alpha$, we use the notation from
\cite[\S 9]{vG}.
Let $\Pic(K)=\ZZ h$ with $h^2=2d$, then the transcendental lattice $T_K$ is isomorphic to
$T_K\cong\ZZ v\oplus\Lambda'$ where $v^2=-2d$
and $\Lambda'\cong U^2\oplus E_8(-1)^2$.
Any $\alpha\in \Br(K)_2=\mbox{Hom}_\ZZ(T_K,\ZZ/2\ZZ)$ can then be written as:

$$
\alpha\colon T_K\,\longrightarrow\, \ZZ/2\ZZ, \qquad (nv,\lambda)\,\longmapsto\, a_{\alpha}n+\lambda_{\alpha}\lambda\mod\;2~,
$$ 

so $\alpha=(a_\alpha,\lambda_\alpha)$ for a unique $a_{\alpha}\in \{0, 1\}$ and 
a $\lambda_\alpha\in \Lambda'$ whose class 
in $\Lambda'/2\Lambda'\cong(\ZZ/2\ZZ)^{20}$ is uniquely determined by $\alpha$.
We will fix an isomorphism $H^2(K,\ZZ)\cong U\oplus\Lambda'$ such that $h=(1,d)_1$,
$v=(1,-d)_1$, both in the first component $U$. Notice that $(0,1)_1\cdot v=1$
and also $(0,1)_1\cdot h=1$.

A B-field lift $B_\alpha\in H^2(K,\QQ)$ of $\alpha=(a_\alpha,\lambda_\alpha)$ is determined by
$2B_\alpha(nv+\lambda)=a_\alpha n+\lambda\lambda_\alpha\mod 2$.  
Thus we can take
$$
B_\alpha\,=\,\big((0,\mbox{$\frac{a_\alpha}{2}$})_1,\mbox{$\frac{\lambda_\alpha}{2}$}\big)\,\in\,
\mbox{$\frac{1}{2}$}(U\oplus\Lambda'),
\qquad \mbox{hence}\quad B_\alpha h\,=\,\mbox{$\frac{1}{2}$}a_\alpha,\quad 
B_\alpha^2\,=\,\mbox{$\frac{1}{4}$}\lambda_\alpha^2~.
$$

The following theorem completes a similar result in \cite[Prop.~9.2]{vG},
the only difference is for the case $h^2=2d$ with $d$ even;
for these $h$ there are also three cases 
(the two isomorphism classes of lattices $\Gamma_\alpha$ with $a_\alpha=0$ 
were not distinguished in \cite{vG}). In case $d=1$ these classes are well understood in
terms of points of order two and theta characteristics on the branch curve of 
the double cover $\phi_h:K\rightarrow\PP^2$, see also \cite{IOOV}.
For a similar classification of Brauer classes of order $p$ for $p>2$ see
\cite{MSTVA}.

The dual lattice of $\Gamma_\alpha$ is
$\Gamma_\alpha^*:=\mbox{Hom}_\ZZ(\Gamma_\alpha,\ZZ)$ and we identify it with
$$
\Gamma_\alpha^*\,=\,\{x\in \Gamma_\alpha\otimes\QQ:\,x\gamma\in\ZZ,\;
\forall\gamma\in \Gamma_\alpha\,\},\qquad x:y\longmapsto xy~,
$$
where the intersection form on $\Gamma_\alpha\subset H^2(K,\ZZ)$ is extended $\QQ$-bilinearly.
The discriminant quadratic form on the (finite) 
discriminant group $\Gamma^{\ast}_{\alpha}/\Gamma_{\alpha}$ is given by
$$
q_\alpha: \Gamma^{\ast}_{\alpha}/\Gamma_{\alpha} \,\longrightarrow \QQ/2\ZZ,\qquad
q_\alpha(x)\,:=\,x^2~.
$$
The isomorphism class of $\Gamma_\alpha$ is determined
by $q_\alpha$ (\cite[Corollary 1.13.3]{N}).

\begin{thm}(Refinement of \cite[Prop.~9.2]{vG}) \label{vG} 
Let $K$ be a K3 surface with $\Pic(K) = \ZZ h$ and $h^2 = 2d$. 
Then for each $d\in \ZZ_{>0}$, the set of lattices $\Gamma_\alpha$,
with $\alpha\in \Br(K)_2$, $\alpha\neq 0$, is partitioned into three isometry classes.

In the case that $4Bh+h^2\equiv 0\mod 4$,
the isomorphism class of $\Gamma_\alpha$ is determined by
$B_\alpha h\in\mbox{$\frac{1}{2}$}\ZZ/\ZZ$ and $B^2\in\mbox{$\frac{1}{2}$}\ZZ/\ZZ$,
otherwise it is determined by $B_\alpha h\in\mbox{$\frac{1}{2}$}\ZZ/\ZZ$ only.

More explicitly, let $\alpha=(a_\alpha,\lambda_\alpha)\in \Br(K)_2 = \Hom(T_K,\ZZ/2\ZZ), \alpha\neq 0$, 
have B-field representative $B_\alpha$, then:
\begin{enumerate}
\item if $a_{\alpha}= 0$ (equivalently $B_\alpha h\equiv 0)$, then
$\Gamma^{\ast}_{\alpha}/\Gamma_{\alpha}= \ZZ/2d\ZZ \oplus \ZZ/2\ZZ \oplus \ZZ/2\ZZ$. 
There are $2^{20} -1$ Brauer classes $\alpha\in \Br(K)_2$ with $B_\alpha h\equiv 0$. 
\\
a) In case $d$ is even, there are two isomorphism classes of such lattices. One class,
has $2^9(2^{10}+1)-1$ elements and is characterized 
by $\lambda_{\alpha}^2 \equiv 0 \mod 4$, equivalently $B_\alpha^2\equiv 0$.
The other class has $2^9(2^{10}- 1)$ elements and 
is characterized by $\lambda_{\alpha}^2 \equiv 2 \mod 4$, 
equivalently $B_\alpha^2\equiv \mbox{$\frac{1}{2}$}$.
\\
b) In case $d$ is odd, these $2^{20}-1$ lattices are isomorphic to each other.

\item if $a_{\alpha}= 1$ (equivalently $B_\alpha h\equiv \mbox{$\frac{1}{2}$}$),
$\Gamma^{\ast}_{\alpha}/\Gamma_{\alpha}= \ZZ/8d\ZZ$.
There are $2^{20}$ Brauer classes $\alpha\in \Br(K)_2$ with $B_\alpha h=1/2$.\\
a) In case $d$ is even, these $2^{20}$ lattices are isomorphic to each other.\\
b) In case $d$ is odd, there are two isomorphism classes of such lattices. One class,
has $2^9(2^{10}+1)$ elements and is characterized 
by $\lambda_{\alpha}^2\equiv 0 \mod 4$,  equivalently $B_\alpha^2\equiv 0$,
the other class has $2^9(2^{10}- 1)$ elements and 
is characterized by $\lambda_{\alpha}^2 \equiv 2 \mod 4$ 
equivalently $B_\alpha^2\equiv \mbox{$\frac{1}{2}$}$.
\end{enumerate}
\end{thm}

\ts 
Only part (1) is not contained in \cite[Prop.~9.2]{vG}. There it is verified that 
$\Gamma^{\ast}_{\alpha}/\Gamma_{\alpha}= \ZZ/2d\ZZ \oplus \ZZ/2\ZZ \oplus \ZZ/2\ZZ$ 
if $a_{\alpha}= 0$ 
(and $\Gamma^{\ast}_{\alpha}/\Gamma_{\alpha}= \ZZ/8d\ZZ$ else). 
The generators of the dual lattice $\Gamma^{\ast}_{\alpha}\subset\Gamma_\alpha\otimes\QQ$
are $v/2d$, $\lambda_\alpha/2$ and $\mu_\alpha$ for any $\mu_\alpha\in \Lambda'$ with
$\lambda_\alpha\cdot\mu_\alpha=1$. One has, with $(a,b,c)\in\ZZ^3$,
$$
q_\alpha(\mbox{$\frac{a}{2d}$}v+\mbox{$\frac{b}{2}$}\lambda_\alpha+c\mu_\alpha)\,=\,
\mbox{$\frac{a^2}{2d}$}+\mbox{$\frac{b^2\lambda_\alpha^2}{4}$}+bc\,\mod\,2\ZZ~.
$$

The two-torsion subgroup (isomorphic to $(\ZZ/2\ZZ)^3$ of 
$\Gamma^{\ast}_{\alpha}/\Gamma_{\alpha}$ is generated by $v/2$, $\lambda_\alpha/2$ and $\mu_\alpha$
and these elements have coordinates $(da,b,c)\in\ZZ$.

Assume $d$ is even. If $\lambda_\alpha^2\equiv 0 \mod 4$,
then on the two-torsion subgroup $q_\alpha$ takes values in $\ZZ/2\ZZ$, whereas if 
$\lambda_\alpha^2\equiv 2\mod 4$, $q_\alpha(\mbox{$\frac{1}{2}$}\lambda_\alpha)\not\in\ZZ/2\ZZ$.
Thus the lattices $\Gamma_\alpha$ are in (at least) 
two distinct isomorphism classes of lattices.
Changing $\lambda_\alpha^2\mapsto \lambda_\alpha^2+4$ gives an isomorphic
discriminant group, just change $c\mapsto b+c$. 
It follows that there are exactly two isomorphism types of lattices for $d$ even,
they are distinguished by the values of $\lambda_\alpha^2\mod 4$, 
or equivalently by $B_\alpha^2$ modulo integers.

Assume $d$ is odd.  If $\lambda_\alpha^2\equiv 0\mod 8$ one has $q_\alpha=\mbox{$\frac{a^2}{2d}$}+bc$
and if $\lambda_\alpha^2\equiv 2\mod 8$ one has 
$q_\alpha=\mbox{$\frac{a^2}{2d}$}+\mbox{$\frac{b^2}{2}$}+bc$. Changing $(a,b,c)\mapsto (a+db,b,a+c)$
maps the first to the second form (use $d^2\equiv 1\mod 8$),
showing that the discriminant groups are isomorphic. Changing
$\lambda_\alpha^2\mapsto \lambda_\alpha^2+4$ is handled as in the $d$ is even case. 
Thus there is a unique isomorphism type if $d$ is odd and this concludes the proof of (1).
\pfs

\subsection{Moduli spaces of twisted sheaves}\label{twistedmoduli}
We recall the fundamental results on moduli spaces of twisted sheaves 
on a K3 surface $K$. The Mukai lattice $\tH(K,\ZZ)$,
of rank 24, and its bilinear form 
$<\cdot,\cdot>$ are defined by
$$
\tilde{H}(K,\ZZ)\,:=\,H^0(K,\ZZ)\oplus H^2(K,\ZZ)\oplus H^4(K,\ZZ),\qquad
<(r,v,s),(r',v',s')>\,:=\,-rs'-r's+vv'.
$$

We consider a B-field $B\in H^2(K,\QQ)$. 
It defines an isometry (where $\tH(K,\QQ):=\tH(K,\ZZ)\otimes\QQ$):
$$
\exp(B):\tH(K,\QQ)\,\longrightarrow\,\tH(K,\QQ),\qquad
x\longmapsto \exp(B)\wedge x,
$$
where $\exp(B)=(1,B,B\wedge B/2)\in \tH(K,\QQ)$, 
and $\wedge$ indicates the cup product on $\tH(K,\QQ)$, so
$$
\exp(B)\wedge (r,v,s)=(r,v+rB,s+B\wedge v+rB\wedge B/2)~.
$$

The B-field $B$ defines a (polarized) Hodge structure $\tH(K,B,\ZZ)$ of weight two 
on the lattice $\tilde{H}(K,\ZZ)$
as follows (\cite[Def.\ 2.3]{HS_EquivTwisted}):
$$
\tH^{2,0}(K,B):=\CC\omega_{K,B},
\qquad\mbox{where}\quad
\omega_{K,B}:=\exp(B)\omega_K=\omega_K+B\wedge \omega_K,
$$
here $\omega_K$ is a basis of $H^{2,0}(K)$ and
$B\wedge \omega_K\in H^4(K,\CC)$ is the cup product of $B$ and $\omega_K$, next one defines:
$$
\tH^{0,2}(K,B):=\overline{\tH^{2,0}(K,B)},\qquad
\tH^{1,1}(K,B):=
\big(\tH^{2,0}(K,B)\oplus \tH^{0,2}(K,B)\big)^\perp.
$$

The isomorphism class of the Hodge structure
$\tH(K,B,\ZZ)$ depends only on the image $\alpha_B$
of $B$ in the
quotient $H^2(K,\QQ)/(NS(K)_\QQ+H^2(K,\ZZ))=\Br(K)$. 

The trivial Hodge substructure in $\tH(K,B,\ZZ)$ is denoted by
$$
NS(K,B):= \{\,x\in\tH(K,\ZZ)\,:\; <x,\omega_{K,B}>=0\,\}.
$$
i.e.\ $NS(K,B)= \tH(K,B,\ZZ)\cap \tH^{1,1}(K,B)$.
Note that if $x=(r,\lambda,s)$ then
$$
<x,\omega_{K,B}>\;=\;
<(r,\lambda,s),(0,\omega_K,B\wedge\omega_K)>\;=\;
-rB\wedge \omega_K+\lambda\wedge \omega_K=
(-rB+\lambda)\wedge \omega_K.
$$
The kernel of the map $H^2(K,\QQ)\rightarrow H^4(K,\CC)$, 
$\mu\mapsto \mu\wedge\omega_K$ is $NS(K)_\QQ$.
From this it follows that if 
$(-rB+\lambda)\wedge \omega_K=0$
then $-rB+\lambda=D$ for some $D\in \Pic(K)_\QQ$,
hence $\lambda=rB+D$ and $(r,\lambda,s)=r(1,B,0)+(0,D,0)+(0,0,s)$. 
A Brauer class $\alpha$ of order $r$ in $\Br(K)$ has a B-field lift $B\in H^2(K,\QQ)$ 
such that $rB\in H^2(K,\ZZ)$ is primitive and then one finds
$$
NS(K,B)\,=\, 
<(r,rB,0)>\oplus NS(K)\oplus H^4(K,\ZZ)~.
$$  
An element $v\in NS(K,B)$ is called a ($B$-twisted) Mukai vector. 

In analogy with the untwisted case, one defines the
transcendental lattice of $\tH(K,B,\ZZ)$ as:
$$
T(K,B)\,:=\,\{x\in \tH(K,B,\ZZ):\;xy=0\qquad\forall y\in NS(K,B)\,\}.
$$
If $\alpha\in \Br(K)=\Hom(T_K,\QQ/\ZZ)$ is the Brauer class defined by $B$ then
there is an isometry of Hodge structures 
(\cite[Prop.\ 4.7]{HuybrechtsTwist}, where $\Gamma_\alpha$, $T(K,B)$
are denoted by $T(X,\alpha)$ and $T(\phi)$ respectively):
$$
\Gamma_\alpha\,\cong\,T(K,B)~.
$$

Let $\alpha\in \Br(K)$ and let $B$ be a B-field representative for $\alpha$. 
Then there is a twisted Chern character $ch^B$ from the K-group of $\alpha$-twisted coherent
sheaves on $K$ to $NS(K,B)$ and the (twisted) Mukai vector
an $\alpha$-twisted sheaf $E$ is defined as $v^B(E):=ch^B(E)\cdot\sqrt{\mbox{td}(K)}$ 
\cite[\S 4]{HuybrechtsTwist}, \cite{HS-Calderaru-conj}.
If $v=(r,\ell,s)$ is
a primitive $B$-twisted Mukai vector with $r>0$ then
the moduli space $M_v(K,B)$ of $\alpha$-twisted sheaves $E$ with $v^B(E)=v$ 
is an irreducible \HK\ manifold of dimension $2+v^2$
which is deformation equivalent to $K^{[n]}$ with $2n=2+v^2$
(see \cite[Thm.\ 3.16]{Y1}). Moreover, if $v^2>0$ then there is an isometry of Hodge structures
(see \cite[Thm.\ 3.19]{Y1})
$$
H^2(M_v(K,B),\ZZ)\,\cong\, v^{\perp},\qquad NS(M_v(K,B))\,\cong\,v^{\perp}\cap NS(K,B)~.
$$
In particular, the transcendental lattice of $H^2(M_v(K,B),\ZZ))$ is $T(K,B)$.

\subsection{Twisted sheaves and Azumaya algebras}\label{tsaa}
Using twisted sheaves, one finds a convenient description for the Azumaya algebra associated
to a $\PP^1$-fibration. 
In the next proposition we do not assume that the Brauer class is non-trivial
so it holds also for $B=0$. In that case, one has $E=\PP(\cV)$ where $\cV$ is a rank two vector
bundle on $K$ and the Azumaya algebra is $\cA=\cV\otimes\cV^\vee$.
The proposition can be generalized to $\PP^{r-1}$-fibrations and locally
free twisted sheaves of rank $r$ (cf.\ \cite[\S 9]{Kollar}).

In Section \ref{BrBNdiv} we show that an exceptional BN divisor $E$ is a $\PP^1$-fibration
over a K3 surface with $K_E^3=12$. Therefore we compute $K_E^3$ in terms of the Mukai vector
of a twisted sheaf.
Proposition \ref{noU-2} then implies that there are Brauer classes on a general 
K3 surface that do not arise from such exceptional divisors.

\begin{prop}\label{uniqueU}
Let $p:E\rightarrow K$ be a $\PP^1$-fibration over a K3 surface 
with Brauer class $\alpha\in \Br(K)_2$ (as in Section \ref{gluingB}) and let $B$ be a B-field representative of $\alpha$.
Then $E=\PP(\cU)$ for a locally free $\alpha$-twisted sheaf of rank two $\cU$ on $K$,
the Azumaya algebra defined by $E$ is isomorphic to $\cU\otimes\cU^\vee$
and one has:
$$
c_2(\cA)\,=\,v^B(\cU)^2 \,+\,8~,\qquad K_E^3\,=\,2c_2(\cA)~.
$$
\end{prop}

\begin{proof}
The conic bundle $E$ is locally trivial (in the complex topology), 
so for a suitable open covering $K=\cup_iU_i$ we have $E_i:=p^{-1}(U_i)\cong U_i\times\PP^1$
and the gluing is by isomorphisms $\phi'_{ij}\in PGL(2,\cO_K({U_i\cap U_j}))$
which can be lifted to $\phi_{ij}\in GL(2,\cO_K({U_i\cap U_j}))$.
These $\phi_{ij}$ define a cocycle $\alpha_{ijk}\in\Gamma(U_i\cap U_j\cap U_k,\cO_K^*)$ which 
represents the Brauer class $\alpha$ of $E$. The $\phi_{ij}$ also define a locally free twisted sheaf
$\cU$ on $K$ with B-field $B$ which also determines $\alpha$ and $E=\PP(\cU)$
(cf.\ \cite[\S 1]{Y1}, \cite[\S 1]{HS_EquivTwisted}).

According to \cite[Thm.~1.3.5]{C} and \cite{dJ}, for any Azumaya algebra $\cA$ over $K$
with two torsion Brauer class $\alpha$ there
is a locally free $\alpha$-twisted sheaf of rank two $\cU$ such that 
$\cA \simeq \mathcal{E}nd(\cU)$. 
More explicitly, let $\cG_i:=p_*\cO_{E_i}(1)$, where $\cO_{E_i}(1)$ is the pull-back of $\cO_{\PP^1}(1)$ 
along the projection $E_i=U_i\times\PP^1\rightarrow \PP^1$. 
Then $\cG_i\cong\cO_{U_i}\otimes H^0(\PP^1,\cO_{\PP^1}(1))$, a locally free sheaf of rank two,
and $E_i=\PP(\cG_i^\vee)$. The $p^*\cG_i$ are locally free sheaves on $E_i$ which glue after twisting by a collection of tautological line bundles to
a locally free sheaf of rank two $\cG$ on $E$ (\cite[\S 1.1]{Y1}). This sheaf sits in the (non-split) 
relative Euler sequence on $E$:
$$
0\,\lra\,\cO_E\,\lra\,\cG\,\lra\,\omega_{E/K}^{-1}\,\lra\,0
$$
(here $\omega_{E/K}^{-1}$ is usually written as $\cT_{E/K}$). Such extensions are parametrized
by $H^1(E,\omega_{E/K})\,(\cong H^1(E,\omega_K)=H^{3,1}(E))$ which is one dimensional, hence the exact sequence
characterizes $\cG$. According to \cite[\S 6]{CI}, where $\cG$ is denoted by $J$, the Azumaya algebra
on $K$ defined by $E$ is $\cA:=p_*\cE nd_E(\cG)=p_*(\cG^\vee\otimes\cG)$.

Using the trace map, $\cA=\cO\oplus\cA_0$
and $\cA_0\cong p_*\omega_{E/K}^{-1}$ \cite[Corollary 6.5]{CI}.
Hence $K_E^3=2c_2(\cA_0)=2c_2(\cA)$ by Proposition \ref{propcw}.

Now from \cite[p.~5]{Y1}  we have $\cU\otimes \cU^{\vee}\simeq \cA$.
Let $ch^B(\cU)=(r,\lambda,s)$, with $r=2$, be the twisted Chern character of $\cU$. Then 
$ch^{-B}(\cU^\vee)=(r,-\lambda,s)$ and thus 
$$
ch(\cA)\,=\,ch^B(\cU)\cdot ch^{-B}(\cU^\vee)\,=\,(r^2,0,2rs-\lambda^2),\qquad
\mbox{hence}\quad c_2(\cA)\,=\,-2rs\,+\,\lambda^2~.
$$
On the other hand, $v^B(\cU)=ch^B(\cU)\sqrt{\mbox{td}(K)}=(r,\lambda,r+s)$, hence
$v^B(\cU)^2=-2r(r+s)+\lambda^2$. 
\end{proof}

\

As a consequence, if $\cU$ is stable then the moduli space of deformations of $\cU$ has dimension
$v^B(\cU)^2+2=c_2(\cA)-6$, which agrees with \cite[Thm 3.6]{HoffmanStuhler}.
In the `extreme' case that $v^B(\cU)^2=-2$ (equivalently, $c_2(\cU\otimes\cU^\vee)=6$),
the following proposition shows that the stability of $\cU$ is automatic and thus $\cU$ is rigid.

\begin{prop}\label{stableU}
Let $(K,h)$ be a polarized K3 surface of degree $h^2=2d$ and assume that $\Pic(K)=\ZZ h$.
Let $\cU$ be a locally free $\alpha$-twisted sheaf of rank two with non-trivial
Brauer class $\alpha\in \Br_2(K)$ represented by a B-field $B$ 
and assume that $v^B(\cU)^2\leq 0$. Then $\cU$ is stable. 
\end{prop}

\ts
First we show that $\cU$ is simple, that is,
$\dim \Hom(\cU,\cU)=h^0(K,\cU^\vee\otimes\cU)=h^0(K,\cA)=1$
where $\cA:=\cU\otimes \cU^\vee$ is the Azumaya algebra associated to the $\PP^1$-fibration
$p:E:=\PP(\cU)\rightarrow K$.
Since $\cA=\cO\oplus \cA_0$ this is equivalent to $h^0(K,\cA_0)=0$. 
As $\cA_0=p_*\omega_{E/K}^{-1}$ and $\omega_{E/K}=\omega_K$,
we get $h^0(K,\cA_0)=h^0(E,-K_E)=0$, by Proposition \ref{h0-KE}. 
By assumption, $v(\cU)^2\leq 0$ so it now follows from \cite[Prop.~3.12]{Y1} 
that $\cU$ is stable. 
\qed

\

We recall that the non-trivial elements $\Br(K)_2$ come in three types.
For $\alpha\in \Br(K)_2$, $\alpha\neq 0$, choose a representative B-field $B$.
Then first of all, we have the invariant $2Bh\in\ZZ/2\ZZ$.
If $4Bh+h^2\equiv 0\mod 4$, then $\alpha$ has the additional invariant $2B^2\in\ZZ/2\ZZ$.

We now show that in case $4Bh+h^2\equiv 0\mod 4$ and $B^2\in\ZZ$ there exists no locally free
rank two $\alpha$-twisted sheaf $\cU$ with $v^B(\cU)^2=-2$. Thus only two of the three
classes of elements in $\Br(K)_2$ have an associated rank four
Azumaya algebra $\cA$ with $c_2(\cA)=6$. In Proposition \ref{uniqueU'} we show that in case
$4Bh+h^2\equiv 2\mod 4$ such a sheaf $\cU$ is unique up to twisting by line bundles.

\begin{prop}\label{noU-2}
Let $(K,h)$ be a polarized K3 surface of degree $h^2=2d$ and assume that $\Pic(K)=\ZZ h$.
Let $\alpha\in \Br(K)_2$ with B-field representative $B$.
There exists a semistable $\alpha$-twisted sheaf $\cU$ of rank two on $K$
with $v^B(\cU)^2=-2$ only in one of the following three cases:
\begin{enumerate}
\item[i)] $\alpha=0$ and $h^2\equiv 2\mod 4$,
\item[ii)] $4Bh+h^2\equiv 2\mod 4$ (in this case $B^2$ is not an invariant of $\alpha$)
\item[iii)] $4Bh+h^2\equiv 0\mod 4$ and $B^2 \notin\ZZ$.
\end{enumerate}
\end{prop}
\ts
In case $\alpha=0$, the Mukai vector of $\cU$ should be $v=(2,kh,s)$ for some integers $k,s$ and thus $v^2=-4s+k^2h^2$. To have $v^2=-2$
one needs $h^2\equiv 2\mod 4$ and then, for any odd $k$ and $s=(2+k^2h^2)/4$, the moduli
space $M_v(K)$ consists of one point, this sheaf $\cU$ satisfies all the conditions.

In case $B\neq 0$, $NS(K,B)$ is generated by $(2,2B,0),h$ and $H^4(K,\ZZ)=\ZZ$.
The Mukai vector of $\cU$ should thus be $v^B=(2,2B+kh,s)$ and
$$
(v^B)^2\,=\,(2,2B+kh,s)^2\,=\,-4s+4B^2+4kBh+k^2h^2~.
$$
If $4Bh+h^2\equiv 2\mod 4$ then one can take $k$ even so that  $4kBh+k^2h^2\equiv 0\mod 4$
and choose
the B-field $B$ representative of $\alpha$ to have $B^2=1/2$ and let
$s:=1+(4kBh+k^2h^2)/4$. Alternatively, take $k$ odd and choose $B$ with $B^2=0$ and
$s:=(2+4kBh+k^2h^2)/4$.

In case $4Bh+h^2\equiv 0\mod 4$, also $4kBh+k^2h^2\equiv 0\mod 4$ and one needs
$B^2\in(1/2)\ZZ$, $B^2\not\in\ZZ$.
\qed

\

A final result concerns the uniqueness of certain twisted sheaves for one of the three
types of Brauer classes.

\begin{prop}\label{uniqueU'}
Let $(K,h)$ be a polarized K3 surface of degree $h^2=2d$ and assume that $\Pic(K)=\ZZ h$.
Let $\alpha\in \Br(K)_2$, $\alpha\neq 0$, have B-field representative $B$
and assume that $4Bh+h^2\equiv 2 \mod 4$.

Let
$\cU$, $\cU'$ be locally free rank two $\alpha$-twisted sheaves on $K$ with 
$v^B(\cU)^2=v^B(\cU')^2=-2$.

\noindent
Then there exists a line bundle $\cL$ on $K$ such that $\cU\cong \cU'\otimes \cL$.
Moreover, if $B^2=1/2$ then there is a line bundle $\cM$ on $K$ such that 
$v^B(\cU\otimes\cM)=(2,2B,1)$.
\end{prop}

\ts
Since $\alpha$ has order two, $NS(K,B)$ is generated by $(2,2B,0),h$ and $H^4(K,\ZZ)=\ZZ$. 
As $\cU$ has rank two, we find $ch^B(\cU)=(2,2B+kh,s)$ for some integers $k,s$.
Let $\cL$ be the line bundle on $K$ with class $-mh$, for some integer $m$,
then $ch(\cL)=1-mh+m^2h^2/2$. Using \cite[Prop.\ 1.2(iii)]{HS_EquivTwisted} 
we get 
$$
ch^B(\cU\otimes\cL)=ch^B(\cU)ch(\cL)=(2,2B+kh,*)(1,-mh,*)=(2,2B+(k-2m)h,*).
$$
Hence choosing $m$ suitably and replacing $\cU$ by $\cU\otimes\cL$  
we may assume that
$ch^B(\cU)=(2,2B,t)$ or $ch^B(\cU)=(2,2B+h,t)$ for some integer $t$.
Then we have $v^B(\cU)=ch^B(\cU)\sqrt{\mbox{td}(K)}=(2,2B,t+2)$ or 
$v^B(\cU)=(2,2B+h,t+2)$. Computing $v^B(\cU)^2$ we find:
$$
-2\,=\,-4(t+2)+4B^2,\qquad \mbox{or}\quad -2\,=\,-4(t+2)+4B^2+4Bh+h^2~.
$$
Since $4Bh+h^2\equiv 2 \mod 4$, the value of $B^2 \mod\ZZ$ is not an invariant. 
If we assume that $B$ is chosen such that $B^2\not\in \ZZ$, then $4B^2\equiv 2 \mod 4$, 
hence the second equation has no solutions. The first equation shows that 
$\cU$ has Mukai vector $v:=(2,2B,t+2)$ with $v^2=-2$ (so with $t=(4B^2-6)/4$).
Assuming $B^2=1/2$ we find $t=-1$ and $v=(2,2B,1)$.

(If instead we assume that $B^2\in\ZZ$ then the first equation has no solution and the second
one shows that $\cU$ has Mukai
vector $v:=(2,2B+h,t+2)$ with $v^2=-2$.)
Given $\cU,\ \cU'$ as in the proposition, there is thus a line bundle, again denoted by $\cL$,
such that $v=v^B(\cU)=v^B(\cU'\otimes\cL)$.
By Proposition \ref{stableU}, $\cU,\ \cU'$ are stable and thus also after tensoring by
a line bundle. Since $M_v(K,B)$ is a point, it follows that $\cU\cong\cU'\otimes\cL$.
\qed

\

\subsection{Brauer classes of degree two K3's}\label{Brauerd2}
In the case of a general
K3 surface $(K,h)$ of degree two and a Brauer class $\alpha\in \Br(K)_2$ 
that corresponds to a point of order two 
(cf.\ \cite{vG}, \cite{IOOV}) one has $h^2=2$, $Bh\equiv 0\mod \ZZ$, so
$4Bh+h^2\equiv 2\mod 4$. Thus there is a unique $\PP^1$-fibration $E=\PP(\cU)$
with Brauer class $\alpha$ having $v^B(\cU)^2=-2$, equivalently, with $K_E^3=12$. 
We exhibited this conic bundle in (the proof of)  Proposition \ref{c2order2}, 
see also \S \ref{verraBN}.

A Brauer class $\alpha\in \Br(K)_2$ that corresponds to an odd theta characteristic has 
$Bh\equiv 1/2\mod \ZZ$, so $4Bh+h^2\equiv 0\mod 4$, and extra invariant $B^2\not\in\ZZ$.
A $\PP^1$-fibration $E\rightarrow K$ with this Brauer class was given in Corollary \ref{c2cubpl}
and it has $K_E^3=12$, so $E=\PP(\cU)$ with $v^B(\cU)^2=-2$.

Finally, a Brauer class $\alpha\in \Br(K)_2$ that corresponds to an even theta characteristic has 
$Bh\equiv 1/2\mod \ZZ$, so $4Bh+h^2\equiv 0\mod 4$, and extra invariant $B^2\in\ZZ$.
In this case there does not exist a $\PP^1$-fibration $E\rightarrow K$ with Brauer class
$\alpha$ which has $K_E^3=12$, equivalently, with $c_2(\cA)=6$ where $\cA$ is the Azumaya algebra
defined by $E$.
Ingalls and Khalid \cite[Theorem 4.3]{IK} found a two dimensional
family of Azumaya algebras of rank four, each with $c_2=8$, representing this Brauer class.
This is thus `the best possible result'.

\

\section{Contractions and Heegner divisors} \label{heegner}

\subsection{Contractions and Heegner divisors}\label{hkbig}
In this section we follow Debarre and Macr\`\i\ \cite{DM} to classify all divisorial contractions on 
a \HK\ fourfold $X$ of K3$^{[2]}$-type (with Picard rank two) 
in terms of sublattices of $H^2(X,\ZZ)$,
or equivalently, in terms of Heegner divisors in certain period spaces.

The following theorem collects the general results on 
certain $-2$-classes on $X$ cf.\ \cite{HTnef}. We write $(\cdot,\cdot)$ for the bilinear form associated
to the BBF-form on $H^2(X,\ZZ)$.
Recall that the base of a general divisorial contraction on a \HK\ fourfold 
is a symplectic surface, so it is a K3 surface or an abelian surface (\cite[Prop.~1.6]{Na}).
From \cite{BM} in the case of $K3^{[2]}$ type hyper-K\"ahler fourfolds it is always a $K3$ surface (see Section \ref{BrBNdiv}).
\begin{thm} \label{mar} Let $X$ be an \HK\ fourfold of K3$^{[2]}$-type, 
$L\in \Pic(X)$, $e=c_1(L)$, $(e,e)=-2$ and 
$(e,h)>0$ for some K\"ahler class $h\in H^2(X,\mathbb{R})$. 
Suppose that $(X,L)$ is general in the locus of deformations of
$X$ which keep $e$ of type $(1,1)$.
Let $k$ be the divisibility of $e$ in $H^2(X,\ZZ)$, so
$$
k=\left \{ \begin{array}{rcl} 2 & \mbox{if} & (e,H^2(X,\ZZ))=2\ZZ, \\
1& \mbox{if} & (e,H^2(X,\ZZ))=\ZZ. 
\end{array} \right. 
$$
Then $h^0(L^{\otimes k})=1$ and the unique effective divisor $E\in |L^{\otimes k}|$ 
is reduced and irreducible (\cite{Mar}).

Moreover, let $F\subset X$ be an effective, reduced and irreducible divisor 
with class $f$ such that $(f,f)<0$.
Then by \cite{D} there exists a sequence of flops from $(X,F)$ to $(X',F')$ 
such that $F'$ is contractible through a projective birational morphism.
\end{thm}

\begin{defi} Let $E$ and $k$ be as above. 
Then $E$ is called a Brill-Noether (BN) 
exceptional divisor if $k=1$ (so $E^2=-2$) and a Hilbert-Chow (HC) exceptional divisor
if $k=2$ (so $E^2=-8$).
\end{defi}

We now consider exceptional divisors $E$ and 
the divisor classes on $X$ inducing the contraction.
Let $H_{2d}$ be a big and nef divisor on $X$,
whose class in $\Pic(X)$ is primitive with BBF degree $H_{2d}^2=2d$ 
and divisibility $\gamma\in\{1,2\}$, 
which contracts a prime divisor $E_{2d}$ with $E_{2d}^2=-2$. 
The Picard group of $X$ thus contains the lattice 
$$
\Pic(X)\,\supset\,\ZZ H_{2d}\,\oplus\,\ZZ E_{2d}\;\cong\;
<2d>\,\oplus\,<-2>~.
$$
Notice that the same sublattice appears for a HC contraction, where $E_{2d}$ is not effective 
but $2E_{2d}$ is the class of the exceptional divisor.

A general deformation $(X',H'_{2d})$ of $(X,H_{2d})$ has $\Pic(X')=\ZZ H_{2d}'$ 
and $H'_{2d}$ is then an ample divisor class with square $2d$ with the same divisibility $\gamma$.
Thus $(X',H'_{2d})$ is an element of the (irreducible, quasi-projective, 20 dimensional) moduli space $\cM^{(\gamma)}_{2d}$ of \HK\ fourfolds of K3$^{[n]}$-type with a polarization of BBF-square $2d$ and divisibility $\gamma$, where in case $\gamma=2$ one has $d\equiv 3\mod 4$.
(see \cite[\S 3]{DM}).

The lattice $H^2(X,\ZZ)$, with the BBF-form, is isometric to the lattice
$$
\Lambda\,:=\,U^3 \,\oplus\, E_8(-1)^2\,\oplus\,<-2>~.
$$
We denote by $\delta$ the generator of the last summand of $\Lambda$, so $\delta^2=-2$
and the divisibility of $\delta$ is two since $(\delta,\Lambda)=2\ZZ$.

The orthogonal group $O(\Lambda)$ acts transitively on the set of primitive elements with fixed
BBF-square $2n$ and fixed divisibility $\gamma$, we fix such an element $h^{\gamma}_{2d}\in\Lambda$. It defines a complex variety
$$
\Omega^{\gamma}_{2d}\,:=\,\{\,x\in\PP(\Lambda\otimes\CC):\; q(x,h^{\gamma}_{2d})\,=\,0,
\; q(x,x)\,=\,0,\;q(x,\bar{x})\,=\,0\,\}.
$$
Given $(X',H')\in \cM^{(\gamma)}_{2d}$, an isometry $H^2(X',\ZZ)\cong \Lambda$
which maps $H'$ to $h^{\gamma}_{2d}$ will map
$H^{2,0}(X)$ to an element, its period, in $\Omega^{\gamma}_{2d}$. In this way we obtain the
period map
$$
\wp^{\gamma}_{2d} \,:\, \cM^{(\gamma)}_{2d}\, \longrightarrow\, \cP_{2d}^{(\gamma)}\,:=\,
O(\Lambda,h^{\gamma}_{2d})\backslash \Omega^{\gamma}_{2d},
$$
where $O(\Lambda,h^{\gamma}_{2d}):=\{g\in O(\Lambda):\,g(h^{\gamma}_{2d})=h^{\gamma}_{2d}\}$
and the period space $\cP_{2d}^{(\gamma)}$ is a quasi-projective variety.

\

Since $H_{2d}$ contracts a divisor, it is not ample and thus
$(X,H_{2d})$ does not define a point in $\cM^{(\gamma)}_{2d}$.
That is, its period point in $\cP_{2d}^{(\gamma)}$ does not lie in the image of the period map.
Notice that the rank of $\Pic(X)$ is at least two.

Let $K\subset\Lambda$ be a rank two primitive sublattice with signature $(1,1)$, containing $h^{\gamma}_{2d}$. Let
$$
\Omega_K\,:=\,\{x\in\Omega^{\gamma}_{2d}\,:\; q(x,k)\,=\,0\quad\forall k\,\in\, K\,\},
\qquad \cD^{(\gamma)}_{2n,K}\,:=\,\mbox{im}(\Omega_K\hookrightarrow \Omega^{\gamma}_{2d}
\rightarrow \cP^{(\gamma)}_{2d})~,
$$
then $\cD^{(\gamma)}_{2n,K}$ is a divisor, called a Heegner divisor, in the period space
$\cP_{2d}^{(\gamma)}$. If $(X',H')\in \cM^{(\gamma)}_{2d}$ maps to a point in $\cD^{(\gamma)}_{2n,K}$, then the Picard lattice of $X'$ contains a copy of $K$.
The (finite) union over such Heegner divisors, where $K^\perp$ has fixed discriminant $-2e$, is
$$
\cD^{(\gamma)}_{2d,2e}\,:=\,\bigcup_{disc(K^\perp)=-2e}\,\cD^{(\gamma)}_{2n,K}\quad(\subset\,
\cP_{2d}^{(\gamma)})~.
$$

Since the period map for smooth compact (not necessarily projective) hyper-K\"ahler
fourfolds is surjective \cite[Theorem 8.1]{Hu1}, 
there exists a fourfold of K3$^{[2]}$-type for any given point
in the period domain $\cP_{2d}^{(\gamma)}$. This fourfold is unique up to flops
by the `Standard Global Torelli theorem'
for fourfolds of K3$^{[2]}$-type, see \cite[Corollary 6.5]{Hu4}, based on \cite{V},
\cite{Markman_Mon1, Markman_Mon2}.

It was proven in \cite[Thm.~6.1]{DM} that the period point of $(X,H)$ is contained in a
divisor of the type listed below (notice that we omit $\cD^{(1)}_{2d,10d}, \cD^{(1)}_{2d,2d/5}$
since then the non-ample divisor
gives a small contraction of a $-10$-class, cf.\ \cite{DM}, proof of Theorem 6.1):

\

If $\gamma=1$
\begin{enumerate}\label{p}
\item for any $d$ in an irreducible component of $\cD^{(1)}_{2d,2d}$
(which parametrizes HC contractions), we denote it by $\cD^{(1)}_{2d,2d,\alpha}$
(and $\cD^{(1)}_{2d,2d}=\cD^{(1)}_{2d,2d,\alpha}$ if $d\not\equiv 0,1\mod 4$),

\item
in case $d\equiv 0,1\mod 4$ in a unique other irreducible component of $\cD^{(1)}_{2d,2d}$
denoted by $\cD^{(1)}_{2d,2d,\beta}$, 

\item in one irreducible component of $\cD^{(1)}_{2d,8d}\subset \cP^{(1)}_{2d}$
denoted by $\cD^{(1)}_{2d,8d,\alpha}$.
\end{enumerate}
 
If $\gamma=2$ in one irreducible component of $\cD^{(2)}_{2d,2d}\subset \cP^{(2)}_{2d}$ 
(such components occur iff $d\equiv 3 \mod 4$). This component is denoted by
$\cD^{(2)}_{2d,2d,\alpha}$.

\

\subsection{Description of the five Heegner divisors}
We are going to show in Theorem \ref{BN} that the general point in each irreducible component 
listed in \S \ref{hkbig} is represented by a \HK\ fourfold which is a  
moduli space of twisted sheaves $M_v(T,B)$ on a polarized K3 surface $(T,h_T)$. 
The Mukai vector $v\in \tH(T,\ZZ)$, the type of the $B$-field $B$ and the degree $h_T^2$
are given in Table \ref{table:heegner}.

The classes of the big and nef divisor $H\in \Pic(M_v(T,B))=NS(T,B)$ and of the exceptional divisor 
$E\in NS(T,B)$ are also given in Table \ref{table:heegner}.
Moreover, certain divisibilities which are essential for identifying the 
(irreducible components of the) Heegner divisors,
equivalently to describe the embedding of lattices $\Pic(X)\hookrightarrow \Lambda$,
are given as well.

\begin{thm}\label{BN}  
A general \HK\ fourfold on one of the five Heegner divisors in Table \ref{table:heegner}
is birationally isomorphic to a moduli space of twisted sheaves $M_v(T,B)$,
where $(T,h_T)$ is a general polarized K3
surface of degree $h_T^2$, with Mukai vector $v\in\tH(T,\ZZ)$, 
B-field $B\in \mbox{$\frac{1}{2}$}H^2(T,\ZZ)$, big and nef divisor class $H\in NS(M_v(T,B))$
and exceptional divisor class $s=E\in NS(M_v(T,B))$ as in Table \ref{table:heegner}.
\end{thm}

\ts 
The proof is given case by case in the proposition given in the last row of Table \ref{table:heegner}.
\pfs

Notice that all non-trivial B-fields $B$ on $T$ in Table \ref{table:heegner} have $B^2\notin\ZZ$. 
In case $4Bh_T+h_T^2\not\equiv 0\mod 4$,
there is only one class of $B$-fields, so the value of
$B^2$ is not important, even if the table lists only the case of $B$-fields with $B^2\not\in\ZZ$.

\begin{table}
\caption{Heegner divisors, Brauer classes and moduli spaces}
\label{table:heegner} 
\centering
{\renewcommand{\arraystretch}{1.5}
\begin{tabular}{|l|c|c|c|c|c|c|}
  \hline
  &$\cD^{(1)}_{2d,2d,\alpha}$&$\cD^{(1)}_{8k,8k,\beta}$ &
  $\cD^{(1)}_{2d,8d,\alpha}$& $\cD^{(1)}_{8k+2,8k+2,\beta}$&$\cD^{(2)}_{8k+6,8k+6,\alpha}$  \\
\hline
Type&HC&BN&BN&BN&BN\\
\hline
$\mbox{div}_{\Lambda}(E)$&2&1&1&1&1\\
\hline
$\mbox{div}_{H^\perp}(E)$&2&2&1&2&1\\
\hline
$h_T^2$&$2d$&$2k$&$2d$&$8k+2$&$8k+6$\\
\hline
B-field&0&$h_TB\,=\,\mbox{$\frac{1}{2}$},\;B^2=\mbox{$\frac{1}{2}$}$&
$h_TB\,=\,0,\;B^2=\mbox{$\frac{1}{2}$}$&0&0\\
\hline
$v$&$(1,0,-1)$&$(2,2B,0)$&$(2,2B,0)$&$(2,h_T,2k)$&$(2,h_T,2k+1)$\\
\hline
$H$&$(0,h_T,0)$&$(0,2h_T,1)$&$(0,h_T,0)$&$(0,h_T,4k+1)$&$(0,h_T,4k+3)$\\
\hline
$s=E$&$(2,0,2)$&$(2,2B,1)$&$(2,2B,1)$&$(2,h_T,2k+1)$&$(2,h_T,2k+2)$\\
\hline
Ref.\ &Prop.\ \ref{2d2d1} &Prop.\ \ref{8d8dBr}&Prop.\ \ref{2d8dBr}&Prop.\ \ref{2d2dBr1(4)}
&Prop.\ \ref{2d2d2}\\
  \hline
\end{tabular}
}
\end{table}

The next five propositions describe the lattices and some properties of the corresponding
\HK\ fourfolds parametrized by the Heegner divisors in Table \ref{table:heegner}.

\begin{prop}\label{2d2d1} 
A general point in the irreducible component $\cD^{(1)}_{2d,2d,\alpha}$
corresponds to a \HK\ fourfold $X$ with Picard lattice isomorphic to
$$
\Pic(X)\,\cong\,\ZZ H_{2d}\oplus\ZZ E_{2d}\,=\,<2d>\,\oplus\,<-2>
$$
and the embedding $\Pic(X)\hookrightarrow\Lambda$ can be chosen as
$$
H_{2d}=(1,d)_1,\quad E_{2d}\,:=\,\delta,\quad
\mbox{hence}\quad 
\mbox{div}_\Lambda(H_{2d})\,=\,1,\quad 
\mbox{div}_\Lambda(E_{2d})\,=\,\mbox{div}_{H_{2d}^\perp}(E_{2d})\,=\,2~.
$$
The transcendental lattice of $X$ is isomorphic to 
$$
T_X\,\cong\,<-2d>\,\oplus\, U^2\,\oplus E_8(-1)^2~,\qquad |\det(T_X)|\,=\,2d~.
$$
The general $X$ is isomorphic to the Hilbert square $K^{[2]}$,
where $(K,h_{2d})$ is a general degree $2d$ K3 surface. Moreover, $H_{2d}$
is induced by $h_{2d}$ and $2E_{2d}$ is the class of the exceptional divisor
of the Hilbert-Chow map $K^{[2]}\rightarrow \sym^2 (K)$, where $\sym^2 (K)=K\times K/ \iota$ and $\iota$ is the involution that permutes the factors.
\end{prop}

\ts
Obviously, for any K3 surface $K$ of degree $2d$,
the Hilbert-Chow map $K^{[2]}\rightarrow \sym^2 (K)$ is an example
of the contractions we consider and it corresponds to the first column of Table \ref{table:heegner}. 
For dimension reasons, the general point of an irreducible component of the Heegner divisor
$\cD^{(1)}_{2d,2d}$ parametrizes these contractions.
\pfs

The second irreducible component $\cD^{(1)}_{2d,2d,\beta}$ of $\cD^{(1)}_{2d,2d}$
exists only for $d\equiv 0,1\mod 4$ and we discuss it 
in Prop.\ \ref{8d8dBr} and \ref{2d2dBr1(4)} respectively.

\begin{prop}\label{8d8dBr} 
Let $d\equiv 0\mod 4$, then
a general point in the irreducible component of $\cD^{(1)}_{2d,2d,\beta}$
corresponds to a \HK\ fourfold $X$ with Picard lattice isomorphic to
$$
\Pic(X)\,\cong\,\ZZ H_{2d}\oplus\ZZ E_{2d}\,=\,<2d>\,\oplus\,<-2>
$$
and the embedding $\Pic(X)\hookrightarrow\Lambda$ can be chosen as
$$
H_{2d}=(1,d)_1,\quad E_{2d}\,:=\,(1,-d)_1+2(1,\mbox{$\frac{1}{4}$}d)_2+\delta,\quad
\mbox{hence}\quad 
\mbox{div}_\Lambda(H_{2d})\,=\,1,\quad 
\mbox{div}_\Lambda(E_{2d})\,=\,1~,
$$
and $\mbox{div}_{H_{2d}^\perp}(E_{2d})=2$.
The transcendental lattice of $X$ is isomorphic to 
$$
T_X\,\cong\,(\ZZ^3,M)\,\oplus\, U\oplus\, E_8(-1)^2~,\qquad
M\,=\,
\left(\begin{array}{ccc} 0&0&2\\0&-\mbox{$\frac{1}{2}$}d&1\\2&1&-2\end{array}\right)~.
$$
In case $d\equiv 0\mod 8$, one has an isomorphism of lattices
$T_X\cong T_S$, where $(S,h_{2d})$ is a general K3 surface of degree $2d$.

The general $X\in \cD_{8k,8k,\beta}^{(1)}$ (with $d=4k$) is birationally
isomorphic to a moduli space of sheaves
$M_v(T,B)$, where $(T,h_T)$ is a polarized K3 surface of degree $h_T^2=2k=d/2\;(\in2\ZZ)$,
the B-field $B\in \mbox{$\frac{1}{2}$}H^2(T,\ZZ)$ satisfies
$$
h_T^2=d/2,\qquad h_{T}B\,=\,\mbox{$\frac{1}{2}$},\quad 
B^2 = \mbox{$\frac{1}{2}$}~,
$$
and one can choose
$$
v\,:=\,(2,2B,0),\qquad H_{2d}\,:=\,(0,2h_T,1),\qquad E_{2d}\,:=\,(2,2B,1)\qquad(\in \tH^2(T,\ZZ))~.
$$

\end{prop}

\ts
The Picard lattice is embedded in
$\Lambda$ with the correct divisibilities to obtain the second component in $\cD^{(1)}_{2d,2d}$
according to \cite{DM}.

Notice that $T_X=\Pic(X)^\perp\subset \Lambda$ is generated by 3 vectors and the `rest' of
$\Lambda$:
$$
g_1:=(-1,4e)_1+2(1,e)_2, \quad g_2:=(1,-e)_2,\quad g_3:=(0,1)_2-\delta,\qquad
U\oplus E_8(-1)^2~.
$$
One easily computes the Gram matrix $M=(g_i\cdot g_j)$
and verifies that $\det(T_X)=\det M=2d$.

In case $d/4$ is even, the vectors $(d/4,1,0),(d/8,1,1)\in\langle g_1,g_2,g_3\rangle$
span a copy of $U$ (but they are not the standard basis)
hence, comparing determinants, $(\ZZ^3,M)\cong <-2d>\oplus\, U$
and $T_X\cong <-2d>\oplus\, U^2\oplus E_8(-1)^2$ which is also the transcendental lattice
of general K3 surface of degree $2d$.

To prove that $X$ is isomorphic to the moduli space of sheaves $M_{v}(T,B)$ 
we recall from \S  \ref{twistedmoduli} that 
$$
NS(T,B)\,=\,<(2,2B,0),(0,h,0),(0,0,1)>\qquad\mbox{so}\quad v,H_{2d},E_{2d}\,\in\,NS(T,B)~.
$$
Since $v^2=2$, the moduli space $M_v(T,B)$ is indeed four dimensional
and one easily verifies that $H_{2d},E_{2d}\in v^\perp=H^2(M_v(T,B),\ZZ)$. 
Moreover, $H_{2d}^2=4h_T^2=2d$, $E_{2d}^2=-2$ and $H_{2d}E_{2d}=0$. 

Next we verify the divisibilities.
Since $(1,0,0)\in v^\perp$ and $H_{2d}(1,0,0)=1$ we have indeed $\gamma:=\mbox{div}_\Lambda(H_{2d})=1$.
Moreover, $E_{2d}(1,0,0)=1$, so $\mbox{div}_\Lambda(E_{2d})=1$. 
Next we notice that (in $\Lambda=v^\perp$)
$$
{H_{2d}^{\perp_\Lambda}}\,=\,v^\perp\cap H_{2d}^\perp \,=\,
\{(2h\eta,\eta,B\eta):\eta\in H^2(T,\ZZ),\;B\eta\in H^4(T,\ZZ)\,\}~.
$$
As $E_{2d}(2h\eta,\eta,B\eta)=-2h\eta\in2\ZZ$, we find
$\mbox{div}_{H_{2d}^\perp}(E_{2d})=2$.

Hence the moduli space $M_v(T,B)$ is indeed a general point in $\cD^{(1)}_{8k,8k,\beta}$
and the Torelli theorem implies that $X$ and $M_v(T,B)$ are birational for
some (general) degree $d/2$ K3 surface $(T,h_T)$.
\pfs

The cases $d=4,8$ of Proposition \ref{8d8dBr},
so the K3 surface $T$ has degree $d/2=2,4$ respectively, will be discussed in 
\S  \ref{fanoplane} and \S \ref{boss16}.

\begin{prop}\label{2d8dBr} 
A general point in $\cD^{(1)}_{2d,8d,\alpha}$
corresponds to a \HK\ fourfold $X$ with Picard lattice isomorphic to
$$
\Pic(X)\,\cong\,\ZZ H_{2d}\oplus\ZZ E_{2d}\,=\,<2d>\,\oplus\,<-2>
$$
and the embedding $\Pic(X)\hookrightarrow\Lambda$ can be chosen as
$$
H_{2d}=(1,d)_1,\qquad E_{2d}\,:=\,(1,-1)_2,\quad \mbox{so}\quad 
\mbox{div}_\Lambda(H_{2d})\,=\,1,\quad 
\mbox{div}_{H_{2d}^\perp}(E_{2d})\,=\,\mbox{div}_\Lambda(E_{2d})\,=\,1~.
$$
The transcendental lattice of $X$ is isomorphic to 
$$
T_X\,\cong\,<-2d>\,\oplus\,<2>\,\oplus\,<-2>\,\oplus\,U\oplus E_8(-1)^2~.
$$

The general $X\in \cD_{2d,8d,\alpha}^{(1)}$ is birationally isomorphic to a moduli space of sheaves
$M_v(T,B)$ where $(T,h_T)$ is a polarized K3 surface of degree $h_T^2=2d$,
the B-field $B\in \mbox{$\frac{1}{2}$}H^2(T,\ZZ)$ satisfies
$$
h_T^2=2d,\qquad h_{T}B\,=\,0,\quad 
B^2\,=\, \mbox{$\frac{1}{2}$}~,
$$
and one can choose
$$
v\,:=\,(2,2B,0),\qquad H_{2d}\,:=\,(0,h_T,0),\qquad E_{2d}\,:=\,(2,2B,1)\qquad(\in \tH^2(T,\ZZ))~.
$$
\end{prop}

\ts
We follow \cite{DM}. Since $\gamma=1$ we may assume that
$H_{2d}=(1,d)_1$, so $H_{2d}$ is in the first component $U$ of $\Lambda$, and
since $H_{2d}^\perp$ is generated by $(1,-d)_1\in U_1$, we have
$$
L\,:=\,H_{2d}^\perp\,=\,<-2d>\,\oplus\,U^2\,\oplus\,E_8(-1)^2\,\oplus\,<-2>~.
$$
We choose $E_{2d}\in L$ to be the 
$$
E_{2d}\,:=\,(1,-1)_2,\qquad \mbox{hence}\quad 
\mbox{div}_L(E_{2d})\,=\,\mbox{div}_\Lambda(E_{2d})\,=\,1
$$
(since for example $E_{2d}\cdot (0,1)_2=1$).
With these definitions, $\ZZ H_{2d}\oplus\ZZ E_{2d}$ is a primitive sublattice of $\Lambda$
and thus we may identify
$$
\Pic(X)\,:=\,\ZZ H_{2d}\oplus\ZZ E_{2d}\,\cong\,<2d>\,\oplus\,<-2>
$$ 
for a general \HK\ fourfold $X$ with period point in $\cD^{(1)}_{2d,8d}$.
The transcendental lattice of $X$ is $\Pic(X)^\perp$, which is generated by $(1,-d)_1,(1,1)_2$ 
and the `rest' of $\Lambda$:
$$
\Pic(X)^\perp\,=\,<-2d>\,\oplus\,<2>\,\oplus\, U\,\oplus\,E_8(-1)^2\,\oplus\,<-2>~.
$$
Thus $|\det(\Pic(X)^\perp)|=8d$ and therefore the period point of these 
\HK\ fourfolds indeed lies in $\cD^{(1)}_{2d,8d}$.

To prove that $X$ is birationally isomorphic to the moduli space of sheaves $M_{v}(T,B)$ 
we recall from \S  \ref{twistedmoduli} that 
$$
NS(T,B)\,=\,<(2,2B,0),(0,h,0),(0,0,1)>\qquad\mbox{so}\quad v,H_{2d},E_{2d}\,\in\,NS(T,B)~.
$$
Since $v^2=2$, the moduli space $M_v(T,B)$ is indeed four dimensional
and one easily verifies that $H_{2d},E_{2d}\in v^\perp=H^2(M_v(T,B),\ZZ)$. 
Moreover, $H_{2d}^2=4h_T^2=2d$, $E_{2d}^2=-2$ and $H_{2d}E_{2d}=0$. 

Next we verify the divisibilities.
Since $(2B)^2=2$, $2B\in H^2(T,\ZZ)$ is primitive
and $BH=Bh=0$ and the embedding of $<h,2B>$ in the K3 lattice is unique up to isometry. 
One easily finds a class $t\in H^2(T,\ZZ)$ such that $th=1$, $tB=0$. 
Then $(0,t,0)\in v^\perp$ and  $H(0,t,0)=1$, so $H$ has divisibility $\gamma=1$ in $v^\perp=\Lambda$.
Moreover, $w:=(-1,0,0)\in v^\perp$ and $E_{2d}w=1$, so $\mbox{div}_\Lambda(E_{2d})=1$. 
Since $w\in v^\perp\cap H_{2d}^\perp=H_{2d}^{\perp_\Lambda}$
we also get $\mbox{div}_{H_{2d}^\perp}(E_{2d})=1$.

Therefore the moduli space $M_v(T,B)$ is indeed a general point in $\cD^{(1)}_{2d,8d,\alpha}$
and the Torelli theorem implies that $X$ and $M_v(T,B)$ are birational for
some (general) degree $2d$ K3 surface $(T,h_T)$.
\qed

\

See Section \ref{D128} for the relation with double EPW sextics in the case $d=1$.

\

\begin{prop} \label{2d2dBr1(4)}
Let $d\equiv 1\mod 4$, we will write $d=4k+1$.
A general point in the second component of $\cD^{(1)}_{2d,2d,\beta}$
corresponds to a \HK\ fourfold $X$ with Picard lattice isomorphic to
$$
\Pic(X)\,\cong\,\ZZ H_{2d}\oplus\ZZ E'_{2d}\,=\,
\left(\ZZ^2,\begin{bmatrix}8k+2&4k+1\\ 4k+1& 2k \end{bmatrix}\,\right)\,\cong\,
\left(\ZZ^2,\begin{bmatrix}-2&1\\ 1& 2k \end{bmatrix}\,\right)~,
$$
where $E_{2d}'=(H_{2d}+E_{2d})/2$ and in the isomorphism
we replaced the basis $e_1,e_2$ by $e_1-2e_2,e_2$.
One has $|\det(\Pic(X))|=4k+1=d$.
The embedding $\Pic(X)\hookrightarrow\Lambda$ can be chosen as
$$
H_{2d}=(1,d)_1,\qquad E'_{2d}\,:=\,(1,0)_1+(1,k)_2~.
$$
The transcendental lattice of $X$ is isomorphic to 
$$
T_X\,\cong\,<-2d>\,\oplus\,U^2\oplus E_8(-1)^2~,
$$
thus $T_X$ is isometric to the transcendental lattice of general K3 surface of degree $2d$.

The general $X\in \cD^{(1)}_{2d,2d,\beta}$ 
is birationally isomorphic to a moduli space of sheaves $M_{v}(T)$ (so $B=0$)
where $(T,h_T)$ is a general K3 surface of degree $h_T^2=2d$ and:
$$
v\,:=\,(2,h,2k),\qquad
H_{2d}\,=\,(0,h,4k+1), \qquad 
E_{2d}\,=\,(2,h,2k+1)\qquad(\in \tH(T,\ZZ))~.
$$
\end{prop}

\ts
The Picard lattice is embedded in
$\Lambda$ with the correct divisibilities to obtain the second component 
$\cD^{(1)}_{2d,2d,\beta}$ of $\cD^{(1)}_{2d,2d}$
according to \cite{DM}:
$$
\mbox{div}_\Lambda(H_{2d})\,=\,\mbox{div}_{\Lambda}(E_{2d})\,=\,1,\quad 
\mbox{div}_{H_{2d}^\perp}(E_{2d})\,=\,2~.
$$

Notice that $T_X=\Pic(X)^\perp\subset \Lambda$ is generated by 
$$
g_1:=(-1,4k+1)_1+(2,2k+1)_2, \quad g_2:=(1,-k)_2,\quad \qquad
U\oplus E_8(-1)^2\,\oplus\,<-2>~.
$$
The sublattice of $T_X$ that is generated by $g_1,g_2,g_3$, where $\ZZ g_3=<-2>$ is the last summand
of $T_X$ above, has Gram matrix $M$ and with the change of basis provided by $S$ one has:
$$
M\,=\,
\left(\begin{array}{ccc} 2&1&0\\1&-2k&0\\0&0&-2\end{array}\right),\qquad
S\,=\,
\left(\begin{array}{ccc} 1&0&1\\k&1&k\\4k&2&4k+1\end{array}\right),\qquad
SM({}^tS)\,=\,
\left(\begin{array}{ccc} 0&1&0\\1&0&0\\0&0&-2d\end{array}\right)~,
$$
hence $T_X\cong <-2d>\oplus\, U^2\oplus E_8(-1)^2$.

To prove that $X$ is birationally isomorphic to the moduli space of sheaves $M_{v}(T)$ with $B=0$
we observe that 
$$
NS(T,0)\,=\,<(1,0,0),(0,h,0),(0,0,1)>\qquad\mbox{so}\quad v,H_{2d},E_{2d}\,\in\,NS(T,0)~.
$$
Since $v^2=2$, the moduli space $M_v(T)$ is indeed four dimensional and since $vH_{2d}=vE_{2d}=0$
we see that $H_{2d},E_{2d}\in v^\perp=H^2(M_v(T),\ZZ)$. 
Moreover, $H_{2d}^2=2d$, $E_{2d}^2=-2$ and $H_{2d}E_{2d}=0$.

To verify the divisibilities,
we choose $t\in H^2(T,\ZZ)$ such that
$th=1$ (this is possible since $H^2(T,\ZZ)$ is unimodular and $h$ is primitive). One verifies
that 
$$
t_1:=(1,2kt,0), \quad t_2=(0,2t,1) \,\in\,v^\perp,\qquad
v^\perp\,=\,\ZZ t_1\oplus \ZZ t_2\,\oplus \{(0,\kappa,0):\,\kappa h=0\,\}~.
$$
Since $H_{2d}(t_1+2kt_2)=1$ we have indeed $\gamma:=\mbox{div}_\Lambda(H_{2d})=1$.
Moreover, $E_{2d}t_1=1$, so $\mbox{div}_\Lambda(E_{2d})=1$. Next we notice that (in $\Lambda=v^\perp$)
$$
{H_{2d}^{\perp_\Lambda}}\,=\,v^\perp\cap H_{2d}^\perp \,=\,<2t_1+(2k+1)t_2>\,\oplus  
\{(0,\kappa,0)\in\tH(T,\ZZ):\,\kappa h=0\,\}~.
$$
As $E_{2d}(2t_1+(2k+1)t_2)\in2\ZZ$, we find
$\mbox{div}_{H_{2d}^\perp}(E_{2d})=2$.

Hence the moduli space $M_v(T)$ is indeed a general point in $\cD^{(1)}_{2d,2d,\beta}$
and the Torelli theorem implies that $X$ is birational to $M_v(T)$ for
some (general) degree $2d$ K3 surface $(T,h_T)$.
\pfs

The case $d=1$ (so $k=0$) in Proposition \ref{2d2dBr1(4)} 
was briefly described by O'Grady in  \cite[\S 5.2]{O2},
his divisor $\bSS_2''$ is $\cD^{(1)}_{2,2,\beta}$.
See also \cite[Example 3.4, \S 7,8]{DHMV}.

\

\begin{prop}\label{2d2d2} Let $d\equiv 3\mod 4$ and write $d=4k+3$.
A general point in $\cD^{(2)}_{2d,2d,\alpha}$
corresponds to a \HK\ fourfold $X$ with Picard lattice 
$$
\Pic(X)\,=\,\ZZ H_{2d}\oplus\ZZ E_{2d}\,\cong\,<2d>\,\oplus\,<-2>
$$
and the embedding $\Pic(X)\hookrightarrow\Lambda$ can be chosen as
$$
H_{2d}=2(1,k+1)_1+\eta,\qquad E_{2d}\,:=\,(1,-1)_2~.
$$
The transcendental lattice of $X$ is isomorphic to 
$$
T_X\,\cong\,<-2d>\,\oplus\, U^2\oplus\, E_8(-1)^2~,
$$
thus $T_X$ is isometric to the transcendental lattice of general K3 surface $S$ of degree $2d$.

The general $X\in \cD^{(2)}_{2d,2d,\alpha}$ 
is birationally isomorphic to a moduli space of sheaves $M_{v}(T)$ (so $B=0$)
where $(T,h_T)$ is a general K3 surface of degree $h_T^2=2d$ and:
$$
v\,:=\,(2,h,2k+1),\qquad
H_{2d}\,=\,(0,h,4k+3), \qquad 
E_{2d}\,=\,(2,h,2k+2)\qquad(\in \tH(T,\ZZ))~.
$$
\end{prop}

\ts
We follow \cite{DM}. Since $\gamma=2$ we may assume that
$H_{2d}=2(1,k+1)_1+\eta$ and choosing $E_{2d}=(1,-1)_2$ we obtain the correct divisibilities
for this irreducible component of the Heegner divisor:
$$
\mbox{div}_\Lambda(H_{2d})\,=\,2,\quad 
\mbox{div}_{H_{2d}^\perp}(E_{2d})\,=\,\mbox{div}_\Lambda(E_{2d})\,=\,1~.
$$
The sublattice
generated by $H_{2d},E_{2d}$ is primitive in $\Lambda$ and hence
it is $\Pic(X)$.

Consider the following vectors in $\Pic(X)^\perp$:
$$
g_1:=(1,-k-1)_1, \quad g_2:=(0,1)_1+\eta,\quad g_3:=(1,1)_2, 
$$
together with the remaining $U\oplus E_8(-1)^2$ they span $\Pic(X)^\perp=T_X$.
Notice that $f_1:=g_2+g_3$ and $f_2:=g_1+(k+1)(g_2+g_3)$ span
a hyperbolic plane and with $f_3:=g_3-(2k+2)f_1-2f_2$ we get:
$$
T_X\,=\,\ZZ f_3\,\oplus\,\ZZ f_1\,\oplus\ZZ f_2\,\oplus\,(U\oplus E_8(-1)^2)\,\cong\,
<-2d>\,\oplus\,U^2\oplus\, E_8(-1)^2~.
$$

To prove that $X$ is birationally isomorphic to the moduli space of sheaves $M_{v}(T)$ with $B=0$
we observe that 
$$
NS(T,0)\,=\,<(1,0,0),(0,h,0),(0,0,1)>\qquad\mbox{so}\quad v,H_{2d},E_{2d}\,\in\,NS(T,0)~.
$$
Since $v^2=2$, the moduli space $M_v(T)$ is indeed four dimensional and since $vH_{2d}=vE_{2d}=0$
we see that $H_{2d},E_{2d}\in v^\perp=H^2(M_v(T),\ZZ)$. 
Moreover, $H_{2d}^2=2d$, $E_{2d}^2=-2$ and $H_{2d}E_{2d}=0$. To verify the divisibilities,
we choose $t\in H^2(T,\ZZ)$ such that
$th=1$ (this is possible since $H^2(T,\ZZ)$ is unimodular and $h$ is primitive). One verifies
that 
$$
t_1:=(1,(2k+1)t,0), \quad t_2=(0,2t,1) \,\in\,v^\perp,\qquad
v^\perp\,=\,\ZZ t_1\oplus \ZZ t_2\,\oplus \{(0,\kappa,0):\,\kappa h=0\,\}~.
$$
From this one finds that $H_{2d} w\in2\ZZ$ for all $w\in v^\perp$ so that indeed
$\gamma:=\mbox{div}_\Lambda(H_{2d})=2$ and $E_{2d}t_1=-(2k+2)+2k+1=-1$, hence 
$\mbox{div}_\Lambda(E_{2d})=1$. Finally we observe that $t_1+(k+1)t_2\in H_{2d}^\perp$
and that $E_{2d}(t_1+(k+1)t_2)=-1$ so that $\mbox{div}_{H_{2d}^\perp}(E_{2d})=1$.
Thus $M_v(T)$ is indeed a general point in $\cD^{(2)}_{2d,2d,\alpha}$
and the Torelli theorem implies that $X$ is birational to $M_v(T)$ for
some (general) degree $2d$ K3 surface $(T,h_T)$.
\pfs

The Heegner divisor $\cD^{(2)}_{2d,2d}$ lies  in the period space 
$\cP^{(2)}_{2d}$ which is non-empty only for
$d\equiv 3\mod 4$. The general \HK\ fourfolds classified by these
spaces are the Fano varieties  of cubic fourfolds in case $2d=6$, they are
Debarre-Voisin \HK\ fourfolds in case $2d=22$ and the Iliev-Ranestad
\HK\ fourfolds in case $2d=38$. We discuss these cases in Sections \ref{nodalcubs}, \ref{DV4f}
and \ref{IR4f}.

\

\subsection{Hilbert squares}\label{exa5.2}
An isometry of lattices $T_X\cong T_S$ does not imply that $X$ is birational to $R^{[2]}$ 
for some K3 surface  $R$. In fact, for this to be true, one needs that $\Pic(X)=T_X^\perp$ contains
a $-2$-class $\delta$ with $\delta^2=-2$ and $\mbox{div}_\Lambda(\delta)\,=\,2$,
see \cite[Rem.~5.2]{DM}.  
We now discuss some cases where we have an isometry $T_X\cong T_S$.
In case $X$ is birationally isomorphic to $S^{[2]}$  
one can study the geometry of the BN divisors also using \cite{KLM}.

\

If $X\in \cD^{(1)}_{2d,2d,\beta}$, cf.\ Proposition \ref{8d8dBr}, and $d\equiv 0\mod 8$, 
there is such an isometry.
Notice that whereas $E_{2d}\in T_X^\perp$ does have $E_{2d}^2=-2$, it has 
$\mbox{div}_\Lambda(E_{2d})\,=\,1$.
From
the description of $T_X^\perp=\Pic(X)=<H_{2d},E_{2d}>\subset \Lambda$ 
it is easy to see that a class $\delta=aH_{2d}+bE_{2d}$ has divisibility $2$ in $\Lambda$ iff
$a\equiv b\mod 2$. If also $-2=\delta^2=2a^2d-2b^2$, that is, $b^2-da^2=1$,
it follows that $a\equiv b\equiv 1\mod 2$. 

In case $d=16$ the Pell equation $b^2-16a^2=1$ has no solutions with $a\equiv b\equiv 1\mod 2$, 
in fact one would obtain $(b-4a)(b+4a)=1$ with $a,b$ odd integers which is impossible,
so in that case $X$ is certainly not birational to the Hilbert square of a K3 of degree $2d=32$.

In case $d=8,24,32$ we do find the solutions $(a,b)=(1,3),(1,5),(3,17)$
to $b^2-da^2=1$ with both $a,b$ odd. It follows that in these cases there is a Hodge isometry
$H^2(X,\ZZ)\cong H^2(S^{[2]},\ZZ)$ for some Hilbert square of a K3 surface $S$ of degree ${2d}$
and thus $X$ is birational to $S^{[2]}$.

\

Similarly, if $X\in \cD^{(2)}_{2d,2d,\alpha}$, cf.\ Proposition \ref{2d2d2},
there is an isometry $T_X\cong T_S$ where $S$ is a
K3 surface of degree $2d=8k+6$.
In that case, a class $\delta=aH_{2d}+bE_{2d}\in \Pic(X)$ has divisibility $2$ in $\Lambda$ iff
$b\equiv 0 \mod 2$. This class is a $-2$-class iff $-2=\delta^2=2a^2d-2b^2$, that is, $b^2-da^2=1$.
In the cases $d=3,11,19$ we do find the solutions $(a,b)=(1,2),(3,10),(39,170)$ with $b$ even.

\

\section{Brauer classes of exceptional BN divisors}\label{BrBNdiv}

\subsection{Brauer classes of exceptional BN divisors}\label{moduli}
In this section we deduce from the work of Bayer and Macr\`\i\ \cite{BM} that the exceptional divisor of a
divisorial contraction on a \HK\ 
fourfold of K3$^{[2]}$ type is a $\PP^1$-fibration over a K3 surface.
Next we determine the Brauer classes
of these fibrations on the exceptional divisors of BN contractions. We conclude with some results on the
relations between the second cohomology groups of the \HK\ fourfold $X$ admitting the
BN contraction with exceptional divisor $p:E\rightarrow K$, of
the $\PP^1$-fibration $E$ and of the K3 surface $K$. In particular, if
the Brauer class $\alpha\in \Br(K)_2$ 
of $E$ is non-trivial, then we identify geometrically 
the sublattice $\Gamma_\alpha$ of index two in the transcendental lattice $T_K$ of $K$.

First we show that if $E$ is a  BN-exceptional divisor, then, remarkably, 
the canonical divisor of $E$ always has self-intersection number $K_E^3=12$ (equivalently,
$c_2(p_*\omega_{E/K}^{-1})=6$).

\begin{prop}\label{degreeE}  Let $p:E\rightarrow K$ be a divisorial 
contraction on a \HK\ manifold $X$ of K3$^{[2]}$ type with $q(E)=-2$.
Then $K_E^3=12$ and thus $c_2(\cW)=6$, where $\cW=\omega_{E/K}^{-1}$. 
Moreover, we have the following intersection numbers:
$$
E(H-E)^3\,=\,12(d-1),\qquad E^2(H-E)^2=4(3-d),\qquad E^3(H-E)=-12~,
$$
where $H$ is the big and nef divisor on $X$ with $q(H)=2d$ which contracts $E$.
In particular, the degree of the linear system $H-E$ on $E$ is $12(d-1)$. 
\end{prop}
\ts
By adjunction on $X$, which has $K_X=0$, we have $K_E=E_{|E}$ hence $K_E^3=E^4$.
The Beauville-Bogomolov form on $H^2(X,\ZZ)$ has the property $x^4=3(x,x)^2$,
so $K_E^3=3(E,E)^2=12$. From Proposition \ref{propcw} we then find $c_2(\cW)=6$.
We also have $x^3y=3(x,x)(x,y)$ (this follows for example by
considering $x+y$ instead of $x$ in the identity $x^4=3(x,x)^2$), so the degree of
$H-E$ on $E$ is $(H-E)^3E=3(H-E,H-E)(H-E,E)=3\cdot(2d-2)(2)=12(d-1)$.
For the last intersection number, $(H-E)^2E^2$, use $x^2y^2=2(x,y)^2+(x,x)(y,y)$.
\pfs 

\

The next theorem shows that a BN divisor $E$ on $X$ is a $\PP^1$-fibration 
$p:E\rightarrow K$ over a K3 surface and thus it naturally defines a Brauer class 
$\alpha\in \Br(K)_2$.
On the other hand, the twisted moduli space structure
on $X=M_v(T,B)$ (cf.\ Thm.~\ref{BN}) depends on  a B-field $B$ which defines a Brauer class
$\beta\in \Br(T)_2$.
The theorem also asserts that $K\cong T$ and that we may identify $\alpha$ and $\beta$.

\

\begin{thm} \label{BNBr}
Let $X$ be a \HK\ fourfold of $K3^{[2]}$ type with Picard rank two and
with a BN divisorial contraction $X\supset E\to K\subset Y$. 
Then $p:E\rightarrow K$ is a conic bundle over a K3 surface.

Let $\alpha\in \Br(K)_2$
be the Brauer class defined by this conic bundle over $K$.
Let $T$ be a K3 surface such that
$X$ is birationally isomorphic with $M_v(T,B)$ as in Thm.~\ref{BN}
and let $\beta\in \Br(T)_2$ be the Brauer class defined by the B-field $B\in H^2(T,\QQ)$.
Then there is an isomorphism $T\cong K$ and any such isomorphism maps $\beta$ to $\alpha$.
\end{thm}

\ts 
Let $s$ be as in Table \ref{table:heegner} and let $\cH:=<v,s>\subset NS(T,B)$.
Since $v^2=2,s^2=-2$ and $vs=0$, ${\mathcal H}\cong
<2>\oplus<-2>$ so ${\mathcal H}$ is an isotropic hyperbolic rank two sublattice. 
One easily verifies that ${\mathcal H}$ is a primitive sublattice (notice that $v-s=(0,0,-1)$).
Then, by \cite[Thm.~5.7]{BM},
there is a Bridgeland stability condition $\sigma$ such that the corresponding
moduli space $M_{v,\sigma}(T,B)$ is isomorphic to $X$.
The exceptional divisor $E\subset M_{v,\sigma}(T,B)$ is described in the proof of \cite[Lem.~8.8]{BM},
it is a $\PP^1$-fibration over $K:=M^{st}_{v-s,\sigma_0}(T,\beta)$.
Since $v-s=(0,0,-1)$, $K=M_{v-s}(T,B)$ and $K$ parametrizes skyscraper sheaves
${\mathcal{O}}_{P}[-1]$ with $P\in T$, this gives an isomorphism $T\cong K$.

The contraction $E\rightarrow K$ is described in \cite[Lem.~8.7]{BM}.
It involves the unique 
Gieseker semistable  torsion free twisted sheaf such that $M_s(K,B)=\{\cS\}$,
with $s\in NS(T,B)$ as in Table \ref{table:heegner}.
If $\cS$ is not locally free, then $\cQ:=\cS^{\vee\vee}/\cS$ 
has support in a finite number of points. Let $s=v^B(\cS)=(r,\lambda,t)$, then
$v^B(\cS^{\vee\vee})=(r,\lambda,l+n)$ where $n$ is the length of $\cQ$.
Thus $$v^B(\cS^{\vee\vee})^2=v^B(\cS)^2-2rn.$$ Notice that $v^B(\cS)^2=-2$ and that  
$\cS^{\vee\vee}$ is also semistable, hence $v^B(\cS^{\vee\vee})^2\geq -2$. Thus $\cQ=0$
and $\cS$ is locally free.
From Table \ref{table:heegner} one finds that $v(\cS)=(2,*,*)$ in all cases under consideration,
hence $\cS$ is a locally free $\beta$-twisted sheaf of rank $2$.
The $\PP^1$ fiber in $E\subset M_{v,\sigma}(K,B)$ over $P\in K$ parametrizes
the surjections $\cS\to \cO_P$. 
Thus we find that $E=\PP(\cS)$ and $p:E\rightarrow K$ is induced by the bundle projection 
$\cS\rightarrow K$.

The conic bundle $E=\PP(\cS)\rightarrow T\cong K$ is defined by 
the $\beta$-twisted sheaf $\cS$ on $T$. Hence the Brauer class defined by $E$ is
$\beta\in \Br(T)_2$.
Thus $\beta$ maps to $\alpha$ since both are determined by the $\PP^1$-fibration $E$. \pfs

\

We explicitly state the following corollary of Theorem \ref{BNBr},
it implies Theorem \ref{mainthm}.

\begin{cor}\label{ko}
Let $X$ be a \HK\ fourfold of K3$^{[2]}$ type admitting a BN contraction with exceptional
divisor $E\subset X$ and let $p:E\rightarrow K$ be the induced $\PP^1$-fibration on a K3
surface $K$. Assume that $\Pic(K)=\ZZ h$ and that $h^2=2d$.

If the Brauer class $\alpha\in \Br(K)_2$ of $E$ is trivial then $h^2\equiv 2\mod 4$.
Let $h^2=8k+2$ or $8k+6$, then
$X \in \cD_{8k+2,8k+2,\beta}^{(1)}$, $X\in  \cD_{8k+6,8k+6,\alpha}^{(1)}$ respectively.
In both cases $E\cong\PP(\cU)$ where the locally free rank sheaf $\cU$ of rank two is a Mukai bundle, so it is stable with $v(\cU)^2=-2$,
in fact we may assume $v(\cU)=(2,h,m)$, with $m=2k+1,2k+2$ respectively.

If the Brauer class $\alpha\in \Br(K)_2$ of $E$ is non-trivial then 
$X \in \cD_{2d,8d,\alpha}^{(1)}$ or $X\in  \cD_{8d,8d,\beta}^{(1)}$. 
The class $\alpha$ has a B-field representative
with $B^2=1/2$ and $Bh=1/2$ in the first case and $B^2=1/2$, $Bh=0$ in the second case.
In both cases $E\cong\PP(\cU)$ where the $\alpha$-twisted locally free sheaf $\cU$ of rank two is stable with  $v^B(\cU)^2=-2$.
\end{cor}

\ts
If $X$ admits a BN contraction and $\Pic(K)=\ZZ h$ we see from
the proof of Theorem \ref{BNBr} that $E=\PP(\cU)$ for a locally free rank two $\alpha$-twisted
sheaf $\cU$ with $v^B(\cU)=s$ with $B$ and $s$ as in Table \ref{table:heegner}. 

In case $\alpha=0$, it is well-known
(cf.\ \cite[Thm.\ 3]{M1} and see \cite[10.3.1]{Huybrechts-K3} for references and the proof)
that there is a unique stable rank two bundle on $K$ with Mukai vector
$s=[E]=(2,h,m)$ with $m=2k+1,2k+2$ respectively
as in Table \ref{table:heegner}, so $s^2=-2$, called the Mukai bundle.

In case $\alpha$ is non-trivial, Table \ref{table:heegner} shows that it has
a B-field representative of $\alpha$ with $B^2=1/2$
and $Bh=0,1/2$ respectively. The proof of Theorem \ref{BNBr} shows that $v^B(\cU)^2=-2$.
The stability of $\cU$ follows from Proposition \ref{stableU}.
\pfs

\subsection{The second cohomology groups}
We consider again a general BN contraction of a \HK\ fourfold $X$ as in Table \ref{table:heegner}.
Then there is a K3 surface $T$ with B-field $B$ and Mukai vector $v\in NS(T,B)$ such that
$X=M_v(T,B)$ and the base of the associated $\PP^1$-fibration $p:E\rightarrow K$
is the K3 surface $K=M_{v-s}(T,B)\cong T$ (see Theorem \ref{BNBr} and its proof).
Recall that $s=[E]\in H^2(X,\ZZ)$ is the class of the exceptional divisor.

Since both $X$ and $K$ are moduli spaces of twisted sheaves on $T$, 
we can explicitly relate the Hodge structures on their second cohomology groups
using a natural map $r'$, defined in the proof of the following Proposition \ref{Brcon},
$$
r':\,s^\perp\cap H^2(X,\ZZ)\;\longrightarrow\; H^2(K,\ZZ)~.
$$
We will give a geometrical interpretation of $r'$ in \S \ref{O'Grady's r}.

We also show that the image of $r'$ is a sublattice of $H^2(K,\ZZ)$
of index two. Since $H^2(K,\ZZ)$ is selfdual, 
there is a B-field $B_X\in \mbox{$\frac{1}{2}$}H^2(K,\ZZ)$ such that 
$$
r'\big(s^\perp\cap H^2(X,\ZZ)\big)\,=\,\{\kappa\in H^2(K,\ZZ):\,B_X\kappa
\equiv 0\mod \ZZ\,\}~.
$$
The Brauer class $\alpha_X\in \Br(K)_2=\mbox{Hom}(T_K,\mbox{$\frac{1}{2}$}\ZZ/\ZZ)$ 
defined by $B_X$,
$$
\alpha_X:T_K\,\longrightarrow\, \mbox{$\frac{1}{2}$}\ZZ/\ZZ,\qquad \tau\longmapsto B_X\tau
$$
is shown to be the same as
the Brauer class $\alpha\in \Br(K)_2$ defined by the conic bundle $p:E\rightarrow K$.
In particular, $\alpha=\alpha_X=0$ iff $r'(T_X)=T_K$.

\begin{prop} \label{Brcon}
Let $X=M_v(T,B)$ and $K=M_{v-s}(T,B)$ as above, then the image of $r'$ is a sublattice of index
two in $H^2(K,\ZZ)$. Let $B_X\in \mbox{$\frac{1}{2}$}H^2(X,\ZZ)$ 
be the B-field that defines the image
of $r'$ and let $\alpha_X\in \Br(K)_2$ be the Brauer class defined by $B_X$.
Then $\alpha_X=\alpha$, where $\alpha$ is
the Brauer class defined by the conic bundle $p:E\to K$. 
\end{prop}
\ts 
In $\tilde{H}(T,B,\ZZ)$ we have the 
Mukai vector $v$ with $v^2=2$ and the
sublattice (and Hodge substructure)
$v^{\perp}\cong H^2(X,\ZZ)$ which contains the class $s$ of $E$ with $s^2=-2$ (and $vs=0$)
(notice that we now consider $s\in \tH(T,B,\ZZ)$ rather than in $H^2(X,\ZZ)$).
Thus $(v-s)^{\perp}\subset \tilde{H}(T,B,\ZZ)$ contains the isotropic vector $v-s$, and 
there is a Hodge isometry
$(v-s)^{\perp}/<v-s>\cong H^2(K,\ZZ)$.
Notice that $s^{\perp} \cap v^{\perp}\subset (v-s)^{\perp}$.
From this we get a Hodge isometry
$$
r':\,s^\perp\cap H^2(X,\ZZ)\,=\,s^\perp\cap v^\perp\,\hookrightarrow\,(v-s)^\perp\,
\longrightarrow\,(v-s)^{\perp}/<v-s>\cong H^2(K,\ZZ)~.
$$

For BN contractions, the sublattice of $\tH(T,\ZZ)$ 
generated by $v$ and $s$, which is isometric to $<2>\oplus<-2>$, is primitive 
and we embed it
in a sublattice $U^2\subset
\tH(T,\ZZ)$ so that $v=(1,1)_1$ and $s=(1,-1)_2$.
Next we define $v^*:=(0,1)_1$, $s^*:=(0,1)_2$
so that $v,v^*$ and $s,s^*$ span the two orthogonal copies of $U$ and $ss^*=1$, $vv^*=1$. 
One verifies that
$$
(v-s)^\perp\,=\,<v-s>\oplus <s+2s^*> \oplus <v^*+s^*>\oplus (U^2)^\perp~.
$$
Hence we can, and will, identify $(v-s)^\perp/(v-s)$ with a sublattice of $\tH(T,\ZZ)$:
$$
H^2(K,\ZZ)\,=\,(v-s)^\perp/(v-s)\cong<s+2s^*> \oplus <v^*+s^*>\oplus (U^2)^\perp\,
\cong\,U\oplus (U^2)^\perp~.
$$
In particular, $H^2(K,\ZZ)$ is identified with this sublattice.

Since
$$
s^\perp\cap H^2(X,\ZZ)\,=\,s^\perp\cap v^\perp\,=\,<v-2v^*>\oplus <s+2s^*> \oplus (U^2)^\perp~,
$$ 
the generator $v-2v^*=(v-s)+(s+2s^*)+2(v^*+s^*)\in s^\perp\cap v^\perp$ 
maps to 
$$
r'(v-2v^*)\,=\,(s+2s^*)+2(v^*+s^*)\,\in(v-s)^\perp/(v-s)~, 
$$
whereas $r'(s+2s^*)=s+2s^*$.
Hence
$$
r'\big(s^\perp\cap H^2(X,\ZZ)\big)\,=\,<s+2s^*> \oplus <2(v^*+s^*)>\oplus (U^2)^\perp
\;\subset\; (v-s)^\perp/(v-s)~,
$$
showing that we indeed get a sublattice of index two. 

Since $s+2s^*\in (v-s)^\perp$ and
$$
\mbox{$\frac{1}{2}$}(s+2s^*)(s+2s^*)\,=\,1,\quad
\mbox{$\frac{1}{2}$}(s+2s^*)(v^*+s^*))=\mbox{$\frac{1}{2}$},
$$
we see that the intersection product with $\mbox{$\frac{1}{2}$}r'(s+2s^*)$ takes integral values on
the index two sublattice, but not on all of $H^2(K,\ZZ)$ and thus we can take 
$B_X=\mbox{$\frac{1}{2}$}r'(s+2s^*)$.

In view of Theorem \ref{BNBr}, $\alpha_X=\alpha$ follows from $\alpha_X=\beta$. 
Recall that $\beta$ has B-field representative $B\in H^2(T,\QQ)$
listed in Table \ref{table:heegner} and also $s,v\in \tH(T,\ZZ)$ are given there.
We explicitly give $s^*$ for each column. Since $<s,v,s^*>$ 
is isometric to $U\oplus<2>$, we can also find a $v^*$, but we won't need it.

We determine the class of 
$B_X\mod H^2(K,\ZZ)+\mbox{$\frac{1}{2}$}\Pic(K)$, with
$B_X=\mbox{$\frac{1}{2}$}r'(s+2s^*)$.
Notice that in all four columns we have $v-s=(0,0,-1)$ so that
$$
(v-s)^\perp=H^2(T,\ZZ)\oplus H^4(T,\ZZ),\qquad (v-s)^\perp/(v-s)\,=\,H^2(T,\ZZ)\,\cong H^2(K,\ZZ)~.
$$

For the first two columns, we choose $s^*=(-1,0,0)$, then $s+2s^*=(0,2B,1)$. 
Under the effective Hodge isometry 
$H^2(T,\ZZ)\,\cong H^2(K,\ZZ)$ the B-field $B_X=\mbox{$\frac{1}{2}$}r'(s+2s^*)$ thus maps to 
$B\in \mbox{$\frac{1}{2}$}H^2(K,\ZZ)$ and so the Brauer classes
$\alpha_X$, $\beta$ (defined by $B_X$, $B$ respectively) are the same. 

For the last two columns we choose a $t\in H^2(T,\ZZ)$ such that $ht=1$ and moreover $t^2=0$. 
Then $h,t$ generate a copy of $U$ in $H^2(T,\ZZ)$. 

For the third column, the case of $\cD^{(1)}_{8k+2,8k+2,\beta}$,
let $s^*=(-1,-2kt,0)\in \tH(T,\ZZ)$ then $<s,s^*>$ 
is a copy of $U$ which is orthogonal to $v$.
Since $v-s=(0,0,-1)$ we can again identify $(v-s)^\perp/(v-s)=H^2(T,\ZZ)=H^2(K,\ZZ)$
and denoting $r'(h),r'(t)$ by $h,t$, we get:
$$
\mbox{$\frac{1}{2}$}r'(s+2s^*)=\mbox{$\frac{1}{2}$}r'(0,h-4kt,2k+1)=\mbox{$\frac{1}{2}$}h-2kt
\,\in\,\Pic(K)_\QQ\,+\, H^2(K,\ZZ)~.
$$
Thus $B_X\equiv 0\in \Br(K)_2$.

For the last column, the case $\cD^{(2)}_{8k+6,8k+6,\alpha}$,
we choose  $s^*=(-1,-(2k+1)t,0)\in \tH(T,\ZZ)$ and as for the third column we find
$$
\mbox{$\frac{1}{2}$}r'(s+2s^*)=\mbox{$\frac{1}{2}$}r'(0,h-(4k+2)t,2k+2)
=\mbox{$\frac{1}{2}$}h-(2k+1)t
\,\in\,\Pic(K)_\QQ\,+\, H^2(K,\ZZ)~.
$$
Thus again $B_X\equiv 0\in \Br(K)_2$.
\qed

\

\subsection{O'Grady's map $r$}\label{O'Grady's r}
We recall some results on the relation between the integral
cohomology of the three manifolds $X,E$ and $K$, following O'Grady \cite[3.9, 4.7]{O2}.
The conic bundle $p:E\rightarrow K$ does not have a section in general and the restriction map
$\Pic(E)\rightarrow \Pic(E_q)=\ZZ$ where $E_q:=p^{-1}(q)\cong\PP^1$ 
is the fiber of $p:E\rightarrow K$ over $q\in K$, has image $2\ZZ$.
However, since we can deform $(X,E)$ to the case that $E=\PP(\cV)$ for a rank two vector bundle $\cV$
on $K$, there is a class in $H^2(E,\ZZ)$ that restricts to a generator of $H^2(E_q,\ZZ)=\Pic(E_q)$.
Choose $\eta_p\in H^p(E,\ZZ)$ such that $i^*\eta_p$ generates $H^j(\PP^1,\ZZ)$
for $j=0,2$ where $i:\PP^1=E_q\rightarrow E$ is the inclusion of a(ny) fiber in $E$.

The Leray-Hirsch theorem (see \cite[Theorem 4D.1]{Hatcher}, \cite[Theorem 7.33]{voisin})
states that there is an isomorphism, where $\eta$ is a linear combination of the $\eta_j$:
$$
H^*(K,\ZZ)\otimes H^*(\PP^1,\ZZ)\,\longrightarrow\, H^*(E,\ZZ),\qquad 
\xi\otimes i^*\eta\,\longmapsto\,p^*\xi\wedge \eta~,
$$
In particular, the Betti numbers of $E$ are $h^i(E)=1,0,1+22=23,0$ for $0,\ldots,3$.
Notice that in general this isomorphism does not respect the Hodge structures 
since a general $\PP^1$-fibration over $K$ does not have a section.
Using the injectivity of $p^*$ one can still determine the Hodge numbers of $E$:
$$
h^{0,0}(E)=1,\quad h^{1,0}(E)=0,\quad h^{2,0}(E)=1, \quad h^{1,1}(E)=21,\quad
h^{3,0}(E)=h^{2,1}(E)=0~.
$$

Using deformations and Hodge theory, \cite[Cor.\ 3.25.3]{O2} shows that
$i^*:H^2(X,\ZZ)\rightarrow H^2(E,\ZZ)$ is injective, notice that both groups have rank $23$.
Moreover,
$$
s^\perp \stackrel{\cong}{\longrightarrow}\,i^*(s^\perp)\,\subset\, 
c^\perp\,:=\,\{\gamma\in H^2(E,\ZZ):\;c\cdot\gamma=0\, \}\,=\,p^*H^2(K,\ZZ)~,
$$
where $s\in H^2(X,\ZZ)$ is the class of $E$, and $c\in H^4(E,\ZZ)$ is the class of a fiber of $p$,
\cite[Cor.\ 3.25.2]{O2}, the last equality follows from the Leray-Hirsch theorem.
Thus one obtains an injective homomorphism $r$ which is shown to preserve the quadratic forms, so
$\tau^2=r(\tau)^2$ for all $\tau\in s^\perp$,
\cite[Claim 3.26]{O2}:
$$
r:\,s^\perp\,\longrightarrow\,H^2(K,\ZZ)~.
$$
Finally, using that $s$ is a class with $s^2=-2$, one finds that $s^\perp$ is a lattice 
with discriminant $4$ whereas $H^2(K,\ZZ)$ is unimodular and hence \cite[(449)]{O2}:
$$
[H^2(K,\ZZ)\,:\,r(s^\perp)]\,=\,2~.
$$

We will now show that $r=r'$ and thus we obtain a more geometrical description 
of the map $r'$ in the proof of Proposition \ref{Brcon}.

\begin{prop} The natural isometry $r(s^\perp)\rightarrow r'(s^\perp)$ extends uniquely
to an isometry $H^2(K,\ZZ)\rightarrow H^2(K,\ZZ)$, hence the maps $r$ and $r'$ may be identified.
\end{prop}

\ts Since both $r$ and $r'$ are isometries between $s^\perp$ and its image, there is a natural
isometry $r(s^\perp)\rightarrow r'(s^\perp)$. Using the correspondence between 
isotropic subgroups of the discriminant group and even overlattices of $s^\perp$,
O'Grady \cite[Claim 4.6]{O2} showed that there is a unique even unimodular overlattice of $s^\perp$.
Using the maps $r,r'$ that overlattice can then be identified with $H^2(K,\ZZ)$.
\qed

\section{Examples of Heegner divisors and BN contractions}\label{exaHeegner}

\subsection{The Heegner divisor $\cD^{(1)}_{2,8}$ and double EPW sextics}
\label{D128}\label{verraBN}
A general element of $\cD^{(1)}_{2,8}$ corresponds to a \HK\ fourfold 
which is birational to a moduli space $M_{(2,2B,0)}(K_2,B)$ 
for a K3 surface $(K_2,h_2)$ of degree two with Picard group $\Pic(K_2)=\ZZ h_2$, 
as in Table \ref{table:heegner}.
The B-field $B$ is non-trivial and satisfies
$Bh_2=0$, $B^2=1/2$, since $4Bh_2+h_2^2=2$ is not divisible by four, 
the value of $B^2$ is not an invariant.
Equivalently, the Brauer class $\alpha=\alpha_B\in \Br(K_2)_2$ defined
by $B$ has $a_\alpha=0$, $\lambda_\alpha^2=2$. 
Since $h_2^2=2$, there are $2^{20}-1$ such Brauer classes (see Thm.\ \ref{vG}.1b).

The Brauer class $\alpha$ corresponds to a line bundle $\cL=\cL_\alpha\in \Pic(C)_2$ 
of order two on the sextic branch curve $C\subset\PP^2$ of
the double cover $K_2\rightarrow \PP^2$ (see \S \ref{Brauerd2}, \cite{vG}, \cite{IOOV}).
In Proposition \ref{c2order2} we constructed a
conic bundle $p:E=E_\alpha\rightarrow K_2$
with this Brauer class and with $K_E^3=12$.
We recall that the push-forward $\cL$ to $\PP^2=\PP^2_y$ has a resolution which determines a
symmetric $3\times 3$ matrix $M(y)$ whose entries are
homogeneous polynomials of degree two as in the proof of Prop.\ \ref{c2order2}.
This matrix defines a Verra threefold, a conic bundle $V'_K$ over $\PP^2$
(with singular fibers over $C$):
$$
V'_K\,:=\,\{(x,y)\in \PP^2_x\times\PP^2_y\,:\;{}^txM(y)x\,=\,0\,\}~.
$$
The pull-back of this conic bundle along $g$ to $K_2$ is birationally isomorphic by contracting one of the irreducible divisors over $C$ to the conic bundle $p:E\rightarrow K_2$ whose Brauer class is $\alpha$.

\

There is another description of these \HK\ fourfolds, as resolutions of 
singular double EPW sextics,
and the conic bundles due to O'Grady in \cite{O2}. 

Let $V$ be a six dimensional complex vector space and let $\mL G(\wedge^3V)$
be the Lagrangian Grassmannian parametrizing maximally isotropic (for the wedge product) subspaces
in $\wedge^3V$. A general $A\in \mL G(\wedge^3V)$ defines an EPW sextic hypersurface 
$Y_A\subset\PP V$
which is singular along a surface of degree $40$. There is natural double cover 
$X_A\rightarrow Y_A$
which is a \HK\ fourfold of K3$^{[2]}$ type, called an EPW sextic.
In general has $\Pic(X_A)=\ZZ H_A$ and $H_A^2=2$. 

Let $\Sigma\subset \mL G(\wedge^3V)$ be the divisor of those $A$ which for which there exists
a three dimensional subspace $W\subset V$ such  that $\wedge^3W\subset A$. For a general
$A\in \Sigma$ there is a unique such $W$.
In \cite[Cor.\ 3.17]{O2}
it is shown that for general $A\in\Sigma$ the EPW sextic $Y_A$ is singular
along a K3 surface $K_2$ (denoted by $S_A$ in \cite{O2}) 
and that $X_A$ has a \HK\ desingularization
$\tilde{X}_A$ with a big and nef divisor $\tilde{H}_A$ which is a 
deformation of a general $(X_A,H_A)$ (\cite[Corollary 3.21]{O2}).
The map $\varphi_{\tilde{H}}:\tilde{X}_A\rightarrow\PP V$ defined by $\tilde{H}_A$ 
is the composition of a BN contraction $\tilde{X}_A\rightarrow X_A$
and a degree two map $X_A\rightarrow Y_A$. In the diagram below $U_1$ is the unique
three dimensional subspace of $V$ with $\wedge^3U_1\in A$.
$$
\begin{array}{rccccl}\varphi_{\tilde{H}}\colon& \tilde{X}_A& \longrightarrow& X_A&
\stackrel{\mbox{\tiny 2:1}}{\longrightarrow}& Y_A \subset \PP V\cong\PP^5\\
&\cup&&\cup&&\cup\\
&E_2&\longrightarrow&K_2&\stackrel{g}{\longrightarrow}&\PP U_1 \cong\PP^2~.
\end{array}
$$
The map 
$g:K_2\rightarrow\PP U_1 \subset \PP V\cong \PP^5$
is induced by the degree two map $X_A\to Y_A\subset\PP V$ 
and $(K_2,h_2)$ is a K3 surface of degree $2$
where $h_2=g^*\cO_{\PP U_1}(1)$.
The double cover $g$ is branched along 
a sextic curve $C=C_A\subset \PP U_1$ which is smooth in general.

\begin{prop} \label{Efor128}
Let $B$ be a B-field representative of the Brauer class of the conic bundle
$E_2\rightarrow K_2$ and let $\alpha=\alpha_B\in \Br(K_2)_2$ be the corresponding Brauer class.

If $\rk\, \Pic(K_2)=1$ then the conic bundles $E=E_\alpha$ and $E_2$ on $K_2$ are isomorphic.
Moreover, the \HK\ fourfolds $\tilde{X}_A$ and  $M_{(2,2B,0)}(K_2,B)$ are isomorphic.

\end{prop}

\ts
From O'Grady's description of the Picard lattice of $\tilde{X}_A$
in $\Lambda$ in \cite[section 4.2]{O2} one finds that
his divisor $\bSS^*_2$ (see \cite[Proposition 4.12]{O2},
it is a covering of degree $2^{20}-1$ of the moduli space of degree
two K3 surfaces by the map that forgets the Brauer class),
is the Heegner divisor $\cD^{(1)}_{2,8}$.
Hence by Table \ref{table:heegner}
we have $Bh=0$, $B^2=1/2$.
By Theorem \ref{BN}, $\tilde{X}_A$ is birationally isomorphic to
$M_{(2,2B,0)}(K_2,B)$, but actually there is an isomorphism
since there are no $-10$ classes of divisibility $2$
in their K\"ahler cones.

By Proposition \ref{uniqueU}, $E\cong \PP(\cU)$ and $E_2\cong \PP(\cU_2)$
for locally free rank two $\alpha$-twisted sheaves on $K_2$. Since $K_E^3=12$ we have
$v^B(\cU)^2=-2$. From  Proposition \ref{degreeE} we get that $K_{E_2}^3=12$, hence also  $v^B(\cU_2)^2=-2$.
Proposition \ref{uniqueU'} implies that $\cU_2\cong\cU\otimes\cL$
for some line bundle $\cL$ on $K$, hence
$E\cong\PP(\cU)\cong\PP(\cU\otimes\cL)\cong\PP(\cU_2)\cong E_2$.
\qed

\

\subsubsection{Remark.} 
Let $V_K$ be the double cover of $\PP^2\times \PP^2$ branched along $V_K'$.
Similarly as in \cite[Thm.~4.5]{IK} one can now show that $\PP^1$-fibration 
$E_2\rightarrow K_2$ is isomorphic to the relative Hilbert scheme of lines  
$\Hilb_{(1,0)}V_K\to K$, where  $\Hilb_{(1,0)}V_K$ is 
the Hilbert scheme of curves in the fibers of the quadric fibration $\pi\colon V_K\to \PP W$ 
of bidegree $(1,0)$ with respect to the two projections of $V_K\subset C(\PP^2\times \PP^2)$ 
to $\PP^2$ (cf.~\cite[\S 4]{Kuz}, \cite[\S 2]{BKK}).
Moreover, we can show that $V_K$ isomorphic to $V_A$,
the Verra fourfold defined from $A$ as in \cite[(2.18)]{KV}.

\subsection{The Heegner divisor $\cD^{(1)}_{8,8,\beta}$ and Fano's of cubic fourfolds with a plane}
\label{fanoplane}
The Fano fourfold $F$ of lines in a smooth cubic fourfold $X$ is a \HK\ fourfold 
with an ample class $g\in \Pic(F)$
defined by the Pl\"ucker map which has BB-square $g^2=6$. 

Let now $X$ be a smooth cubic fourfold with a plane $T\subset X$, such fourfolds are parametrized
by the Hassett divisor $\cC_8$.
Then $T$ defines a divisor (class)
$\tau=\alpha(T)\in \Pic(F)$ where $\alpha:H^4(X,\ZZ)\rightarrow H^2(F,\ZZ)$ is the 
Abel-Jacobi map (which only changes the sign on the primitive part).
The Abel-Jacobi map is induced by the incidence correspondence
$Z$ in $X\times F$, so $Z=\{(x,[l])\in X\times F:\;x\in l\}$ (and $l$ is a line in $X$).
One has $\tau^2=-2$ and $g\tau=2$ (\cite[Example 7.5]{HTratcur}).
Thus $\Pic(F)$ has the sublattice (which is equal to $\Pic(F)$ for general 
cubics with a plane):
$$
K_8\,:=\,<g+\tau,\,\tau>\;=\,<8>\,\oplus\,<-2>~.
$$

From the incidence correspondence one finds that 
the support $\tau$, denoted by $E$, consists of the classes $[l]\in F$ of lines $l\subset X$ 
which meet $T$:
$$
E=\{[l]\in F:\;l\cap T\,\neq\,\emptyset\},\qquad  \tau\,=\,[E]~. 
$$

The divisor $E\subset F$ is well known to be a conic bundle over a K3 surface $K$ of degree two.
Choose a plane $T'\subset\PP^5$ disjoint from $T$.
For $t\in T'$ let $P_t:=<T,t>\subset \PP^5$ be the $\PP^3$ spanned by $T$ and $t$. Then 
$P_t\cap X=Q_t\cup T$ where $Q_t$ is a quadric in $P_t$. 
In other words, let $\pi\colon \bar{X}\to X$ be the blow up of $X$ along $T$. 
Then $\pi$ is a quadric bundle over $\PP^2=T'$, 
the fiber of $\pi$ over $t$ is isomorphic to $Q_t$.

Any line in $Q_t$ meets $T$ 
and conversely any line $l\subset X$ meeting $T$ is contained in a $Q_t$.
Hence  the (one or two) rulings of $Q_t$ give one or two rational curves in $E\subset F$ 
and these curves can be mapped to $t$.
This map $E\rightarrow T'$ factors over
a K3 double cover $K\rightarrow T'$ branched over sextic curve in $T'$, cf.\ \cite{voisin_tor}. 
$$
E\,\stackrel{p}{\longrightarrow}\,K\,\stackrel{\mbox{\tiny 2:1}}{\longrightarrow}\,
T'\,\cong\,\PP^2~.
$$
The map $p$ is a $\PP^1$-fibration.
See Corollary \ref{c2cubpl} (and its proof)  for this conic bundle, as well as \cite[\S 4]{Kuz}, 
and \cite{MS} for relations between the associated twisted derived category of $K$ and $F$.
We summarize the discussion above in the following proposition.

\begin{prop}\label{fanoD188} 
A general point in the Heegner divisor $\cD^{(1)}_{8,8,\beta}$ is given by $(F,K_8)$,
where $F$ is the Fano fourfold of a general cubic 4-fold with a plane and 
$K_8=<g+\tau,\tau>$.

The conic bundle $p:E\to K$ is a BN contraction on $F$ induced by $H:=g+\tau$.
The type of the Brauer class of this conic bundle 
is characterised by $hB=B^2=\frac{1}{2}$ and it corresponds to an odd theta characteristic on
the sextic branch curve of $K\rightarrow T'\cong\PP^2$.
\end{prop}

\ts
The divisibilities of $g,\tau$  in $H^2(F,\ZZ)$ are $2$ and $1$ respectively. Hence
the divisibility $\gamma$ of $g+\tau$ in $H^2(F,\ZZ)$ is equal to one and thus
$(F,K_8)$ defines a point in $\cD^{(1)}_{8,2e}$.
To find $e$, we need to determine
$|\det (K_8^\perp)|$.  
As $H:=g+\tau$ has divisibility $1$, we have $|\det(H^\perp)|=16$
and since $\tau$ has divisibility $1$ we get, using \cite[(4)]{DM},
that $2e=2|\det(H^\perp)|/1=16$, hence 
$(F,g+\tau)\in\cD^{(1)}_{8,8}$. 

From Table \ref{table:heegner} we indeed obtain that the base of the conic bundle $K$ 
has a polarization $h_K$ of degree two, as we also described above, and
that the Brauer class $\alpha\in \Br(K)_2$ defined by $E$  
has a B-field representative with $hB=B^2=1/2$. The proof of Corollary \ref{c2cubpl}
gives the relation with the odd theta characteristic on the branch curve, see also 
\cite[Prop.\ 4]{voisin_tor}.
\pfs

\subsection{The Heegner divisor $\cD^{(1)}_{16,16,\beta}$ and the BOSS bundle}\label{boss16}

In case $d=8$, the Heegner divisor $\cD^{(1)}_{2d,2d}$ has a unique irreducible component 
parametrizing fourfolds with a BN contraction denoted by $\cD^{(1)}_{16,16,\beta}$. 
As we observed in \S \ref{exa5.2}, a general $X$
in $\cD^{(1)}_{16,16,\beta}$ is birationally isomorphic to the Hilbert square $S^{[2]}$ of
a K3 surface $S$ of degree $16$. 
From Table \ref{table:heegner} we find that the contraction of the exceptional divisor, 
which we now denote by $Z$ instead of $E$, 
$p:Z\rightarrow K$ is a conic bundle with non-trivial Brauer class $\alpha\in \Br(K)_2$
over a quartic K3 surface $(K,h)$. 
Moreover, any B-field $B$ representing $\alpha$ has
$Bh=\frac{1}{2}$ (modulo the integers). 
Since $4Bh+h^2=2\not\in4\ZZ$, $B^2$ does not give an extra invariant
and we may assume $B^2=1/2$. There are $2^{20}$ such Brauer classes in $\Br(K)_2$ 
(cf.\ Theorem \ref{vG}(2)). The conic bundle is uniquely determined by these data, it
is $\PP(\cU)\rightarrow K$ where $\cU$  is
the unique $\alpha$-twisted stable sheaf with $v^B(\cU)=(2,2B,1)$ by Proposition \ref{uniqueU'}.

In \cite[Corollary 9.4]{vG}, the index 2 sublattice
$\Gamma_\alpha$ of the transcendental lattice $T_K$ of $K$ is shown to be 
the transcendental lattice of a K3 surface and it is in fact $T_S$, 
with $S$ as above, cf.\ Proposition \ref{8d8dBr}. 
Alternatively, from the description of 
the K3 surface $K$ as a moduli space of certain 
sheaves on $S$ (\cite[Thm.\ 3.4.8]{IR} \cite{Alzati})
one can also deduce that $T_S$ is isomorphic to an index two sublattice of $T_K$.
In particular, $S$ determines $K$ uniquely, but for a general quartic K3 surface $K$
there are $2^{20}$ choices for $S$.

The BOSS bundles are degree $12$ threefolds $Z\subset\PP^5$
that are conic bundles over quartic surfaces in $\PP^3$.
These conic bundles were discovered in \cite{BOSS} and further studied in \cite{IR},
which we follow (see also \cite{Alzati}, \cite{Kuz}, \cite{DMS}, \cite{MSTVA}). 
We will explicitly construct an embedding 
$Z\hookrightarrow S^{[2]}$
for a K3 surface $S$ degree $16$ in $\PP^4$. We show in Proposition \ref{boss}
that this implies that the BOSS $\PP^1$-fibrations
are isomorphic to exceptional divisors in \HK\ fourfolds.

\subsubsection{An incidence correspondence} 
Let $V_6$ be a six dimensional complex vector space with a symplectic form. 
There is natural decomposition $\wedge^3V_{14}=V_{14}\oplus V_6$ as representations
of the symplectic group of $V_6$. Let $\Sigma=LG(3,V_6)$ be the Grassmannian of 
Lagrangian (i.e.\ maximally isotropic) subspaces in $V_6$. 
There is a (Pl\"ucker) embedding $\Sigma\hookrightarrow\PP V_{14}$.
For $p\in \PP V_6$ we denote by $Q_p\subset \Sigma$ the subvariety of Lagrangian
subspaces containing $p$. Then $Q_p$ is isomorphic to a smooth 3-dimensional quadric
and it is cut out in $\Sigma$ by a four dimensional linear subspace 
$\PP^4_p\subset\PP V_{14}$, so $Q_p=\PP^4_p\cap \Sigma$ (\cite[\S 2.4]{IR}).

We consider the incidence variety (see \cite[p.394]{IR})
$$
J\,:=\,
\{(p,\omega)\,\in\,\PP V_6\times\PP V_{14}^*\,:\; \PP^4_p\subset\PP^{12}_\omega\}~,
$$
here $\PP^{12}_\omega\subset \PP V_{14}$ is the hyperplane defined by
$\omega\in V_{14}^*$. The image of $J$ under the projection $\pi$ to
$\PP V_{14}^*$ is a quartic hypersurface $F=\pi(J)\subset\PP V_{14}^*$.
The singular locus $\Omega$ of $F$ is $9$-dimensional variety
and the fiber of $\pi$ over a point in $F-\Omega$ is a smooth conic (\cite[Prop.2.5.5]{IR}).

\subsubsection{The K3 surfaces $S$ and $K$}
Let $S$ be a K3 surface of degree $16$ with Picard rank one. Mukai proved that
there is, up to projective equivalence, a unique $\PP^9=\PP^9_S\subset \PP V_{14}$
such that $S=\PP^9_S\cap \Sigma$.
The Mukai dual of $S$ is the quartic K3 surface  $K=\PP^3_K\cap F$,
where $\PP^3_K\subset\PP V_{14}^*$ is the dual of $\PP^9_S$.

The BOSS conic bundle over $K$ defined by $S$
is the restriction of $\pi$ to $Z:=\pi^{-1}(K)\rightarrow K$.
$$
\begin{array}{rrcl}\pi\colon& J& \longrightarrow& \pi(J)\,=\,F\quad\subset\;\PP V_{14}^*\\
&\cup&&\cup\\
 &Z&\longrightarrow&K\,=\,\PP^3_K\cap F~.
\end{array}
$$
For $(p,\omega)\in Z\subset J$, we have $\omega\in K\subset\PP^3_K$ 
and $\PP^4_p\subset\PP^{12}_\omega$. 
Dualizing $\{\omega\}\subset \PP^3_K$ we get $\PP^9_S\subset \PP^{12}_\omega$
and thus $\PP^4_p\cap \PP^9_S\subset \PP^{12}_\omega$. In particular
$\dim \PP^4_p\cap \PP^9_S\geq 1$.

There is an embedding $Z\subset\PP V_6$ given by
$Z=\{p\in\PP V_6:\,\dim (\PP^4_p\cap \PP^9_S)\geq 1\,\}$ (\cite[Remark 3]{Alzati})
and the ideal of the image, a threefold of degree $12$, is generated by ten quintics.
This is the classical description of the BOSS bundle (\cite{BOSS}).

\subsubsection{The BOSS bundle as divisor on a Hilbert square.}\label{bosshilb}
Let $(p,\omega)\in Z\subset J$, so $\dim \PP^4_p\cap \PP^9_S\geq 1$.
As $S=\Sigma\cap \PP^9_S$ and $\Sigma\cap \PP^4_p=Q_p$ we see that
$$
S\,\cap\, (\PP^4_p\,\cap\,  \PP^9_S)\,=\,Q_p\,\cap\, (\PP^4_p\cap\PP^9_S)~.
$$
 
If $\dim \PP^4_p\cap \PP^9_S>1$ then, since $Q_p$ is a quadric threefold
in $\PP^4_p$, the surface $S$ will contain a conic which contradicts our
assumption that $\Pic(S)$ has rank one.
Hence $\PP^4_p\cap \PP^9_S$ is a line. This line is not contained in $S$ since $\Pic(S)$ has rank one.
Thus the line $\PP^4_p\cap \PP^9_S$ intersects $Q_p$ and hence $S$ in a zero cycle of degree two.
This gives a map 
$$
\psi:\,Z\,\longrightarrow\,S^{[2]},\qquad (p,\omega) \,\longmapsto S\cap \PP^4_p\cap\PP^{9}_S~.
$$
This map is an embedding and there is a BN contraction on $S^{[2]}$ which induces $\pi:Z\rightarrow
K$.

\

\begin{prop} \label{boss}
A general point in the Heegner divisor $\cD^{(1)}_{16,16,\beta}$ is given by 
$(S^{[2]},K_{16})$ where 
$(S,h)$ is a general K3 surface of degree $16$,  
$K_{16}=\langle 3H-8\xi,\,H-3\xi\rangle$ with $H\in \Pic(S^{[2]})$ the divisor class
defined by $h$ and $2\xi$ the class of the divisor parametrizing non-reduced subschemes.

The map $\psi:Z\rightarrow S^{[2]}$ is an embedding and its image is the
exceptional divisor $E$ of the corresponding BN contraction. 
The $\PP^1$-fibration $Z\rightarrow K$ is over the Mukai dual quartic K3 surface $(K,h)$ of $S$.
The Brauer class of the BOSS bundle has (unique) invariant $Bh=1/2$.
\end{prop}

\ts
As in \cite[Example 5.3]{DM} we compute the nef cone of $S^{[2]}$. 
The Picard group of $S^{[2]}$ is generated by $H,\,\xi$, and $H^2=16$, $H\xi=0$ and $\xi^2=-2$. 
Since the Pell-type equation
$a^2-32b^2=5$ has no solutions mod $5$, 
hence has no solutions at all, the nef cone and the movable cone coincide
and thus $S^{[2]}$ has a unique birational \HK\ model (cf.\ \cite[\S 5]{DM}) and 
in particular it does not admit any flops. 

The movable cone has extremal rays $H$ and $H-8(b/a)\xi$ where 
$(a,b)$ is the solution of 
$a^2-8b^2=1$, with $a,b>0$ and $a$ minimal.
This minimal solution is $(a,b)=(3,1)$, so the nef cone has extremal rays 
$H$ and $3H-8\xi$. The $(-2)$-classes perpendicular to these rays are $\xi$ and $H-3\xi$.
Notice that indeed $(H-3\xi)(3H-8\xi)=3h^2+24\xi^2=3\cdot 16-24\cdot 2=0$
and $(H-3\xi)^2=16-9\cdot 2=-2$.

In particular, the divisor with class $H-3\xi$ 
is contracted by the big and nef class $H_{16}=3H-8\xi$.
Notice that $H_{16}^2=9\cdot 16-64\cdot 2=16(9-8)=16$, as we know should be the case for 
$(S^{[2]},H_{16})$ to be in $\cD^{(1)}_{16,16,\beta}$.

If $\psi$ is not injective, there are distinct $p,q\in \PP V_6$ such that the lines 
$l_p:=\PP^4_p\cap \PP^9_S$ and $l_q:=\PP^4_q\cap \PP^9_S$ are the same. 
Hence $Q_p\cap Q_q$ is not empty.
As the quadrics $Q_p,Q_q\subset \Sigma$ are either disjoint or intersect along a line,
they must then intersect along a line $l$. As $l$ contains the degree two 0-cycle 
cut out by $l_p$ on $S$, we have $l=l_p$. But $l\subset Q_p\cap Q_q\subset\Sigma$
hence $l\subset S$, which again contradicts that the rank of $\Pic(S)$ is one. 
Finally, $\psi$ is an embedding because the map $\PP V_6\ni p\to \PP^4_p\in G(5,14)$ 
is an embedding.

To prove that the $\PP^1$-bundle $Z\subset S^{[2]}$ 
can be contracted we use \cite[Prop.~3.24]{O2},
which asserts that $h^0(S^{[2]},Z)=1$ and $Z^2<0$ for the BB-form.
From Theorem \ref{mar} we then deduce that $Z$ 
can be contracted after a series of flops. 
But we already observed that $S^{[2]}$ does not admit flops and
hence $Z$ can be contracted, so its class is on an extremal ray.

The extremal ray defined by $\xi$ defines the HC contraction $\PP \cT_{S}\rightarrow S$
and since $Z$ is contracted to $K$, a K3 of degree four, it must define the other extremal ray.
From \cite{BOSS} we know that $K_Z^3=12$.
Hence, as in the proof of Proposition \ref{degreeE}, $Z^2=-2$, so
$Z$ defines a BN contraction and
$Z$ has class $H-3\xi\in \Pic(S^{[2]})$.
From Table \ref{table:heegner} it follows that the period point lies in $\cD_{16,16,\beta}^{(1)}$
and the B-field representative for $\beta\in \Br(T)_2$ given there is also the representative
for $\alpha\in \Br(K)_2$ by Theorem \ref{BNBr}.
\pfs

\subsection{The Heegner divisor $\cD_{6,6}^{(2)}$ and nodal cubic fourfolds}\label{nodalcubs}
For $2d=6$, a general point  of the Heegner divisor $\cD^{(2)}_{2d,2d}$
is isomorphic to $S^{[2]}$ where $(S,h)$ is a general K3 surface of degree $6$.
These K3 surfaces and their Hilbert squares
are related to nodal cubic fourfolds and their Fano varieties
cf.\ \cite[4.2, Lemma 6.3.1]{SCF}, \cite[Example 6.4]{DM}, \cite[4.3]{GO}.

Such a Hilbert square admits two divisorial contractions, one is HC
and the other is BN.
From Corollary \ref{ko} the BN divisor is the projectivisation of the rank two,
stable, Mukai bundle
$\cV$ with $v(\cV)=s=(2,h,2)$ and $s^2=-2$ (see Table \ref{table:heegner}).
In particular, the associated Brauer class is trivial.

The K3 surface $S\subset\PP^4$ is the complete intersection of a smooth quadric 
and a cubic hypersurface. Taking this quadric to be a hyperplane section of
the Grassmannian $Gr(2,4)\subset\PP^5$,
in \cite[Lemma~4.5]{GO} it is shown that the Mukai bundle 
$\cV$ can be chosen to be the restriction of the dual of the universal bundle on $Gr(2,4)$ to $S$.

\subsection{The Heegner divisor $\cD_{22,22}^{(2)}$ and Debarre-Voisin fourfolds}\label{DV4f}
Let $(K,h)$ be a K3 surface of degree $22$ and denote by $h$, $2\xi\in \Pic(K^{[2]})$ 
the classes corresponding to $h$ and the class of the divisor parametrizing
non-reduced subschemes respectively.
In \cite{DV}, Debarre and Voisin show that the general \HK\ fourfold
with period in $\cP^{(2)}_{22}$ is the zero locus of a section of $\wedge^3\cE_6$,
where $\cE_6$ is the tautological rank $6$ vector bundle on the Grassmannian $Gr(6,10)$.
Moreover, these varieties are shown to be deformations of $(K^{[2]},H)$, where
$$
H = 10h \,-\, 33\xi,\qquad
\mbox{so}\quad H^2 \,=\, 10^2\cdot 22 +   33^2 \cdot (-2)\, =\, 22(100 - 99) = 22~.
$$
Note that $HD \equiv 0 \mod 2$ for all $D \in H^2(K^{[2]},\ZZ)$, 
hence the divisibility $\gamma$ of $H$ is indeed $2$. 

In \cite[Lemma 3.6]{DV} it is shown that $\phi_H$ is birationally a contraction 
of a divisor $D'$ to a 
K3 surface $Y^3_\sigma$ of degree $22$ (cf.\ \cite[Prop.\ 3.2]{DV}).
In view of the divisibility of $H$, 
Table \ref{table:heegner} shows that this contraction gives a point in $\cD^{(2)}_{22,22,\alpha}$.

In  \cite[Claim, p.79]{DV} it is shown that $D'$ is not reduced, but has multiplicity two
and combining with \cite[Lemma 3.5]{DV} one finds that $D'=2E$ and the class of 
$E\in H^2(S^{[2]},\ZZ)$ is
$$
E\,:=\,3h-10\xi,\qquad\mbox{so}\quad E^2\,=\,-2,\qquad HE\,=\,0~.
$$
By Corollary \ref{ko} the BN divisor $E$ is the projectivisation of the rank two,
stable, Mukai bundle $\cV$ with $v(\cV)=s=(2,h,6)$ and $s^2=-2$.

The Hilbert square $S^{[2]}$ does admit flops, in fact it has the $-10$ class $2h-7\xi$,
and the moving cone  is divided in two chambers. 
Let $\overline{S^{[2]}}$ be the flop of $S^{[2]}$ along this $-10$ class. 
It follows that $H$ induces a BN contraction on $\overline{S^{[2]}}$.
See also \cite{DHOV} for polarized Hilbert squares that are Debarre-Voisin fourfolds.

\subsection{The Heegner divisor $\cD_{38,38}^{(2)}$ and Iliev-Ranestad fourfolds}\label{IR4f}
The general  \HK\ fourfold with period in $\cP^{(2)}_{38}$ is a variety  of sums of powers
$VSP(F,10)$ where $F$ is a cubic threefold in $\PP^5$ and points of $VSP(F,10)$ 
correspond to the ways of writing the equation of $F$ as a sum of ten cubes of linear forms
(\cite{IRspv} and \cite{IRadd}).

In \cite[\S 5]{RV} it is shown that for a general $F\in V_{V-ap}$, where $V_{V-ap}$ is
the divisor in the moduli space
of cubic fourfolds consisting of cubics apolar to a Veronese surface, the variety $VSP(F,10)$
is actually singular along a K3 surface $S$ and its desingularization 
$$
X=\widetilde{VSP(F,10)}\,\longrightarrow \,Y=VSP(F,10)
$$ 
is a divisorial contraction and the exceptional divisor $E$ is a $\PP^1$-fibration over $S$.
From our results we deduce the following which implies that a general point in 
$\cD^{(2)}_{38,38,\alpha}$ corresponds, birationally, to the contraction above.

\begin{prop}
For a general $F\in V_{V-ap}$, the desingularization $X\rightarrow Y$ is a BN contraction and 
$X=\widetilde{VSP(F,10)}$ is birational to $S^{[2]}$
where $S=Sing(Y)$ is a K3 surface of degree $38$. 
The Brauer class of the conic bundle $E\rightarrow S$
is trivial and $E$ is isomorphic to the projectivisation of the Mukai bundle
on $S$ with Mukai vector $s=(2,h,10)$ where $h^2=38$.

The Picard lattice of $S^{[2]}$, which is isomorphic to 
$\ZZ h\oplus\ZZ\xi=<38>\oplus <-2>$, contains the $-10$ class $2h-9\delta$ 
of divisibility $2$ that defines a Mukai flop from $S^{[2]}$ to $\overline{S^{[2]}}$. 
On $\overline{S^{[2]}}$ the
BN class $E=39h-170\xi$ is contracted by the morphism associated to the divisor
$H=170h-741\delta$ of degree $38$.
\end{prop}
\ts 
Since the contraction is given by a movable divisor $H$ with divisibility $\gamma=2$
we conclude from Table \ref{table:heegner}, where all the possible divisorial 
contractions are listed, that we are in the case of $\cD_{8k+6,8k+6}^{(2)}$ and $k=4$.
In $\Pic(S^{[2]})$ we must have $H=ah+b\xi$, with $h$ the class induced by the polarization
of degree $38$ on $S$, and
$38=H^2=38a^2-2b^2$ and $a$ should be even (since $\gamma=2$)
whereas the exceptional divisor $E$ should be 
perpendicular to $H$, with $E^2=-2$
and moreover the divisibility of $E$ is one, so $E=ch_{38}+d\delta$ and $c$ should be odd. 
Thus $38c^2-2d^2=-2$ and $(c,d)$ is a solution of the Pell equation
$d^2-19c^2=1$ whose solutions are generated by $(c,d)=(39,170)$. Then $HE=0$ (with $H\cdot h_{38}>0$)
gives $(a,b)=(170,\pm 19\cdot 39)$, so up to isometry, these are the classes of $H$ and $E$.
\pfs
 
\subsection{Further examples} Recently Benedetti and Song
described the contractions of the HK 4folds in $\cD^{(1)}_{24,24,\beta}$.
These are Debarre-Voisin fourfolds, denoted by $X^\sigma_6$ in their paper. They show in \cite[Theorem 4.19]{BenedettiS}
that they are twisted moduli spaces $M_v(T,B)$ for a degree six K3 surface $T$ with non-trivial B-field satisfying $B^2=Bh=1/2$ \cite[Lemma 4.15]{BenedettiS} and Mukai vector $v=(2,2B,0)$. Moreover, they show that these fourfolds have a divisorial contraction \cite[Proposition 4.7]{BenedettiS}, recovering our results in a special case.

\section{The Heegner divisor $\cD^{(1)}_{4,16}$, EPW sextics  
and a special EPW sextic}
\label{exBOSS}\label{grzexa}

\subsection{The hyper-K\"ahler fourfolds in $\cD^{(1)}_{4,16,\alpha}$}
Let $(X,H)$ be a general hyper-K\"ahler fourfold of K3$^{[2]}$-type with an ample divisor
class $H$ such that $q(H)=4$. Then   
$$
\varphi_{H}\colon X\longrightarrow Y\subset \PP^9
$$ 
is a birational morphism as can be checked in the case where  $X=S_{4}^{[2]}$,
where $(S,h)$ is a degree $4$ polarized K3 surface, and $H$ is defined by $h$.
Indeed, the map given by $H$ factorises through the HC contraction $S^{[2]}\to \sym^2(S)$
and an embedding $\sym^2(S)\subset \sym^2(\PP)\subset \PP^9$, so it is birational.

Now we assume that $X$ has Picard rank two, that $H$ is only big and nef
and moreover that it induces a BN contraction.
According to Table \ref{table:heegner}, the period point of $(X,H)$ then lies in 
$\cD^{(1)}_{4,16,\alpha}$.
Let the contraction be given by
$X\supset E\to K\subset Y\subset \PP^9$,
then $K$ is a K3 surface of degree $4$. 

The Brauer class of $E$ in $\Br(K)_2$
is non-trivial for the general $(X,H)$ (see Proposition \ref{2d8dBr}).

We show in \S \ref{div416} that these \HK\ fourfolds are also double EPW sextics. 
The natural involution on the fourfold $X$ permutes two conic bundles, 
one of which is $E$, the other will be denoted by $E'$. Both conic bundles
turn out to be birational to the complete intersection of the EPW sextic (in a $\PP^5$) with a quadric,
but this intersection is not a normal variety.

Since the general case remains somewhat mysterious, we study in \S \ref{exaHSd4}
a codimension one subvariety of the divisor $\cD^{(1)}_{4,16,\alpha}$ 
which parametrizes Hilbert squares of degree four
K3 surfaces with a degree four rational curve. 
In that case $E$ is a trivial conic bundle but certain aspects of the geometry are still interesting.
Finally, in \S \ref{Fermat} we consider the concrete case of a certain rational quartic curve on
the quartic Fermat surface in $\PP^3$. It turns out that then the
singular surface of the EPW sextic (in $\PP^5$, see \cite{O2})
has 60 isolated singular points,
which makes this EPW sextic rather special.

\subsection{The divisor $\cD^{(1)}_{4,16,\alpha}$ and double EPW sextics.}\label{div416}
In this section we provide some information on the $-2$-divisor $E\subset X$ which is 
a conic bundle over a quartic surface $K$.
Since the period point of $X$ lies in $\cD^{(1)}_{4,16,\alpha}$, 
the Picard group is, with now $\Lambda=H^2(X,\ZZ)$:
$$
\Pic(X)\,=\ZZ H\,\oplus\,\ZZ E\,=\,\,<4>\,\oplus\,<-2>~,\quad \
\gamma:=\mbox{div}_\Lambda(H)\,=\,1,\quad \det(\Pic(X)^{\perp_\Lambda})=16~.
$$

The nef cone of $X$ has extremal rays spanned by $H$ and $H':=3H-4E$:
$$
\overline{\mbox{Nef}(X)}\,=\,\RR_{\geq 0}H\,+\,\RR_{\geq 0}H',\qquad H':=3H-4E,\quad
H^2\,=\,(H')^2\,=4~.
$$
see \cite[\S 13]{BM}.

The perpendiculars of the extremal rays are
$$
H^\perp=\ZZ E,\quad (H')^\perp=\ZZ E',\qquad E'\,:=\,2H-3E~,\qquad E^2\,=\,(E')^2\,=\,-2~.
$$
Hence $\phi_H$, $\phi_{H'}$ contracts $E$, $E'$ respectively.

Notice that $X$ has the ample class $L:=H-E$ of square two and thus $\phi_L:X\rightarrow\PP^5$
exhibits $X$ as a double EPW sextic. The covering involution $\iota$ induces an involution 
$\iota^*$ on $\Pic(X)$ which is minus the reflection in the 
orthogonal complement of $L$.
$$
\phi_L:\,X\,\longrightarrow\,Y\,:=\,X/\iota\,\subset\,\PP^5,\qquad L\,:=\,H-E~.
$$
The involution $\iota^*$ interchanges the two extremal rays, 
and also $E$, $E'$ are exchanged. Thus these two divisors, which are conic bundles
over quartic surfaces, are mapped 
to the same threefold $F=\phi_L(E)=\phi_L(2H-3E)\subset\PP^5$, of degree $12$ 
by Proposition \ref{degreeE}.

Notice that the sum of the classes of the two $-2$-divisors is $E+E'=2L$.
Since $\phi_L(X)=Y$, an EPW sextic, $Y$ is not contained in any
quadric and thus the pull-back map $\phi_L^*:H^0(\cO_{\PP^5}(2))\rightarrow H^0(2L)$ is injective
and also surjective for dimension reasons. Thus there is a quadric $Q\subset \PP^5$ 
which cuts out the image $F$ of the $-2$-divisors in $Y$. 
As this intersection is a threefold of degree 
$2\cdot 6=12$ and also $F$ has degree 12 we conclude that $F=Q\cap Y$.
By adjunction we compute the canonical sheaf $\omega_F =\mathcal{O}_F(2)$.
Now since $\phi_L|_E$ does not contract curves and the canonical divisor 
$K_E=E|_E\neq 2 L|_E$ we conclude that the image $F$ of $E$ is not normal.

\subsection{Example: Hilbert squares of certain degree four K3's}\label{exaHSd4}
We construct a more explicit example of fourfolds, actually Hilbert squares, 
in $\cD^{(1)}_{4,16,\alpha}$.
Let $(S,h)$ be  a degree $4$ K3 surface with a $(-2)$-curve $n$ so that $hn=4$.
Notice that the quartic curve $n\subset \PP^3$ lies on a quadric surface
and that the residual curve $n'$ is again a rational quartic with class $n':=2h-n$.
The linear system $h_{12}:=h+2n$ has degree 
$$
h_{12}^2\,=(h+2n)^2 =\, h^2\,+\,4hn\,+\,4n^2
\,=\,4+16-8\,=\,12,\qquad
h_{12}n\,=\,4-4=0~.
$$
So in general $\phi_{h_{12}}:S\rightarrow S_{12}\subset\PP^7$
contracts the rational curve $n$ and $S_{12}$ is a nodal degree $12$ K3 surface.

Let $H_{12}$ be the divisor class on the Hilbert square $S^{[2]}$ defined by $h_{12}$ on $S$ and let
$$
H\,:=\,H_{12}-2\xi,\qquad H^2\,=\,12-8\,=\,4~.
$$

The divisor $n$ on $S$ defines the divisor  $E:=n+S$ which is birational to $\PP^1\times S$.
The class of $E$ is $[E]=n=(n,0)\in H^2(S^{[2]},\ZZ)=H^2(S,\ZZ)\oplus\ZZ\xi$, 
where $2\xi$ is the divisor on $S^{[2]}$ parametrizing the non-reduced subschemes.
Using the BB-form we have:
$$
N\,:=\,[E]\,=\,[n+S]~,\qquad
HN\,=\,0,\qquad N^2\,=\,-2~.
$$

Consider the subgroup 
$$
K\,:=\,\ZZ H\,\oplus\,\ZZ N\qquad(\subset H^2(S^{[2]},\ZZ))~,
$$
with the BB-form on $K$ we find the lattice $K\cong <4>\oplus <-2>$.

\begin{lem}\label{S4K} The pairs
$(S^{[2]},K)$ are parametrized by a codimension $1$ subvariety of the irreducible
component $\cD^{(1)}_{4,16}$ of the Heegner divisor $\cD^{(1)}_{4,16}$.
The line bundle $L$ on $S^{[2]}$ with $L^2=2$ from \S \ref{div416} is given by
$$
L\,=\,H\,-\,E\,=\,H_{10}\,-\,2\xi\;\in\,\Pic(S^{[2]})~,
$$
where $H_{10}$ is induced by $h_{10}:=h+n\in \Pic(S)$.
\end{lem}

\ts  
We may assume that $h,n\in H^2(S,\ZZ)=U^3\oplus E_8(-1)^2$ are in the  first two components
$U^2$ and are given by
$$
h\,=\,(1,2)_1,\qquad n\,=\,(0,4)_1+(1,-1)_2\qquad (\subset U^3\oplus E_8(_1)^2)~.
$$
Then $h_{12}=h+2n=(1,10)_1+(2,-2)_2$ and in $\Lambda=H^2(S,\ZZ)\oplus\ZZ\xi$ 
we have $K=<H,N>$ with
$$
H\,=\,H_{12}\,-\,2\xi\,=\,(1,10)_1\,+\,2(1,-1)_2\,-\,2\xi,\qquad N\,=\,(0,4)_1+(1,-1)_2~.
$$
Since $H^2=4$ and the divisibility of $N$ in $\Lambda$ is one (since $N\cdot (0,1)_2=1$)
we see from Table \ref{table:heegner} that $(S^{[2]},K)$ lies in $\cD^{(1)}_{4,16}$.

Recall that $L=H-E=H_{12}-N-2\xi$ in $\Pic(S^{[2]})$. 
Since $H_{12}$ is induced by $h+2n\in \Pic(S)$, $H_{12}-N$ is induced by $h+n$.
\pfs

\subsection{The Fermat quartic and a special EPW sextic}\label{Fermat}
We consider the following smooth rational quartic curve $n$ in $\PP^3$:
$$
n\;:=\;\{\,(s^4:\,st^3:\,s^3t:\,t^4)\,\in\,\PP^3:\;(s:t)\,\in\,\PP^1\,\}~.
$$
This curve lies in the smooth quartic surface 
$$
S:\qquad f\,:=\,z_1^4\,+\,z_2^4\,-\,z_0z_3(z_0^2+z_3^2)\,=\,0~. 
$$
Notice that $S$ is (projectively) isomorphic to the quartic Fermat surface (substitute
$z_0:=z_0+z_3,\,z_3:=z_0-z_3$).

The intersection of $S$ with the quadric
$$
Q:\qquad q\,:=\,z_0z_3\,-\,z_1z_2\,=\,0
$$
consists of $n$ and another smooth rational degree four curve $n'$ 
obtained by permuting the second and third coordinates ($z_1\leftrightarrow z_2$) on $n$:
$$
S_4\cdot Q\,=\,n\,+\,n'~,\qquad n'\,:=\,\{\,(s^4:\,s^3t:\,st^3:\,t^4)\,\in\,\PP^3\,\}~.
$$
Thus $2h_4=n+n'$ and $h_{10}=h+n=3h-n'$. So we see
that $h_{10}$ is given by the restriction to $S$ of the cubics on $\PP^3$ containing $n'$.
Notice that these cubics are independent of $S$ and depend only on the choice of $n'$.

The ideal of $n'$ is generated by the quadratic polynomial $q$ and three cubics: 
$$
f_1\,:=\,-z_0^2z_2 \,+\, z_1^3,\qquad
f_2\,:=\,z_0z_2^2 \,-\, z_1^2z_3,\qquad
f_3\,:=\,-z_1z_3^2 \,+\, z_2^3~.
$$
Thus the map $\phi_{h_{10}}$ is defined as follows:
$$
\phi_{h_{10}}\,:\,S\,\longrightarrow\,S_{10}\,\subset\PP^6,\qquad
\phi_{h_{10}}\,=\,(z_0q:\,z_1q:\,z_2q:\,z_3q:\,f_1:\,f_2:\,f_3)~.
$$
Notice that $\phi_{h_{10}}(n)$ is the conic 
$C_n\subset S_{10}$ defined by $y_0=\ldots=y_3=y_4y_6+y_5^2=0$.

\subsubsection{The image $Z$ of $\PP^3$ in $\PP^6$}\label{3foldF}
The coordinate functions of the map $\phi_{h_{10}}$ do not depend on 
the quartic surface $S_4$ but only on the quartic curve $n'$. The image of $\PP^3$ under
$\phi_{h_{10}}$ is (according to Magma) a smooth threefold $Z$ of degree
$5$ in $\PP^6$ which is the intersection of
five quadrics defined by the polynomials:
$$
F_0\,:=\,- y_0y_5\,+\,y_1^2\,-\, y_2y_4,\quad
F_1\,:=\,y_0y_2\,-\, y_1y_5\,-\, y_3y_4,\quad
F_2\,:=\,- y_1y_2\,+\,y_4y_6\,+\,y_5^2,
$$
$$
F_3\,:=\,- y_0y_6\,+\,y_1y_3\,+\,y_2y_5,\quad
F_4\,:=\,- y_1y_6\,+\,y_2^2\,+\,y_3y_5~.
$$
The birational inverse of $\phi_{h_{10}}:\PP^3\rightarrow Z$ is the projection of $Z$ given by
$$(y_0:\ldots:y_6)\mapsto (y_0:\ldots:y_3),$$ its base locus in $Z$ is the conic $C_N$.

\subsubsection{The Del Pezzo quintic threefold.}
We recall the definition from \cite{MU} (cf.\ \cite[Section 5.1]{KPS})
of a quintic Fano threefold $Z_{MU}$ of
index $2$  in $\PP^6$, known as the Del Pezzo quintic threefold,
which is also a linear section of $Gr(2,5)\subset\PP^9$.
We show that it is isomorphic to $Z$. (More intrinsically this could be done
by studying the projection of $Z_{MU}$ from a conic to $\PP^3$, the inverse map is then
given by the cubics in the ideal of a degree four rational curve.)

The variety $Z_{MU}$ is defined by the following $5$ quadratic polynomials (\cite[p.505]{MU}):
$$
A_0\,:=\,h_0h_4 - 4h_1h_3 + 3h_2^2,\quad
A_1:=h_0h_5 - 3h_1h_4 + 2h_2h_3,\quad
A_2:=h_0h_6 - 9h_2h_4 + 8h_3^2,
$$
$$
A_3:=h_1h_6 - 3h_2h_5 + 2h_3h_4,\quad
A_4:=h_2h_6 - 4h_3h_5 + 3h_4^2~.
$$
To verify that $Z\cong Z_{MU}$ it suffices to verify that for $j=0,\ldots,4$ one has:
$$
F_j(18h_1,-3r^3h_2,-3r^4h_4,-18r^2h_5,r^2h_0,12rh_3,18h_6)\,=\,324A_j(h_0,\ldots,h_6)~,
\qquad (r^5\,=\,18)~.
$$

\subsubsection{The image $S_{10}$ of $S$}
The K3 surface $S_{10}=\phi_{h_{10}}(S)\subset\PP^6$ is the intersection of 
$Z=\phi_{h_{10}}(\PP^3)$ with a sixth quadric. 
In the case we consider here, one can take the quadric defined by
$$
F_5\,:=\,y_0^2 + y_1y_4 + y_2y_6 + y_3^2~.
$$

\subsubsection{The map from $S^{[2]}$ to $\PP^5$}\label{mapX2}
From now on we identify the K3 surfaces $S,S_{10}$.
Recall the rational map   
$$
\phi_L:\,S^{[2]}\,\longrightarrow\,Y\quad(\subset\,\PP^5)
$$
given by 
the divisor class $L=H_{10}-2\xi$,
where $H_{10}$ corresponds to $h_{10}\in \Pic(S)$, so $H_{10}^2=10$, and
$2\xi$ is the divisor parametrizing non-reduced subschemes.
It is defined by $H-2\xi$ is given by
associating to $p+q\in S^{[2]}$ the hyperplane in the $\PP^5$ of quadrics defining $S$
which vanish on the line spanned by $p$ and $q$. This map can be given explicitly as
in the proof of Proposition 2.2 of \cite{DGKKW}: if $p,q\in S$ and $F$ is a quadratic
polynomial vanishing on $S$, then $F(\lambda p+\mu q)=2\lambda\mu B(p,q)$ for a bilinear
form $B$. The bilinear forms defined by the six defining quadrics are the coordinate
functions of a rational map $S\times S \rightarrow\PP^5$ which factors over $S^{[2]}$ to give
the map defined by $L$.

\subsubsection{The EPW sextic $Y$ in $\PP^5$}\label{epw}
The map $\phi_L$ still defines a (rational) 2:1 map onto an EPW sextic in $\PP^5$. 
We checked this explicitly by using the bilinear forms $B_j$ associated to the following 
six quadratic forms $F_j$ (notice the non-standard ordering):
$$
F_5,\;F_0,\;F_1,\;F_2,\;F_4,\;F_3~.
$$

We used Magma for the computations below. The image of $S^{[2]}$ under $\phi_L$ is indeed
a sextic fourfold $Y$ whose singular locus has dimension $2$ and degree $40$, 
as expected for an EPW sextic. The explicit equation has $53$ terms (the coordinates on $\PP^5$ are 
$(x_0:\ldots:x_5)$):
{\renewcommand{\arraystretch}{1.3}
$$
\begin{array}{l}
-4x_0^3x_1x_3x_4\,+\,x_0^3x_1x_5^2\,+\,x_0^3x_2^2x_4\,+\,x_0^3x_2x_3x_5\,+\,x_0^3x_3^3\,+\,
    4x_0^2x_1^3x_4\,-\,3x_0^2x_1^2x_2x_5\,-\,3x_0^2x_1^2x_3^2 \\
    \,+\,3x_0^2x_1x_2^2x_3\,+\,4x_0^2x_1x_4^3\,-\,x_0^2x_2^4\,-\,3x_0^2x_2x_4^2x_5\,-\,
    3x_0^2x_3^2x_4^2\,+\,3x_0^2x_3x_4x_5^2\,-\,x_0^2x_5^4 \,+\,3x_0x_1^4x_3\\
    \,-\,2x_0x_1^3x_2^2\,+\,6x_0x_1^2x_3x_4^2\,-\,6x_0x_1^2x_4x_5^2\,-\,6x_0x_1x_2^2x_4^2\,-\, 
    12x_0x_1x_2x_3x_4x_5\,+\,8x_0x_1x_2x_5^3\\
    \,-\,4x_0x_1x_3^3x_4\,+\,3x_0x_1x_3^2x_5^2\,+\,8x_0x_2^3x_4x_5\,+\,3x_0x_2^2x_3^2x_4\,+\,
    x_0x_2x_3^3x_5\,+\,3x_0x_3x_4^4\,-\, 2x_0x_4^3x_5^2 \\
    \,-\,x_1^6\,-\,3x_1^4x_4^2\,+\,4x_1^3x_2x_4x_5\,+\,4x_1^3x_3^2x_4\,-\, 
    2x_1^3x_3x_5^2\,-\,6x_1^2x_2^2x_3x_4\,-\,3x_1^2x_2x_3^2x_5\,-\,3x_1^2x_4^4 \\
    \,+\, 4x_1x_2^4x_4\,+\,8x_1x_2^3x_3x_5\,+\,x_1x_2^2x_3^3\,+\,4x_1x_2x_4^3x_5\,+\,     
    4x_1x_3^2x_4^3\,-\,6x_1x_3x_4^2x_5^2\,+\,4x_1x_4x_5^4\\\,-\,4x_2^5x_5 
    \,-\,x_2^4x_3^2\,-\,
    2x_2^2x_3x_4^3\,-\,3x_2x_3^2x_4^2x_5\,+\,8x_2x_3x_4x_5^3\,-\,4x_2x_5^5 \,+\, 
    x_3^3x_4x_5^2\,-\,x_3^2x_5^4\,-\,x_4^6~.
\end{array}
$$
}

\subsubsection{The trivial conic bundle}
The divisor $E=n+S\subset S^{[2]}$ is the limit of a conic bundle in the general 
$X\in \cD^{(1)}_{4,16,\alpha}$. The image $F=\phi_L(E)$ 
is the image of $C_n\times S_{10}\subset(\PP^6)^2$,
notice that $C_n\times S_{10}\subset C_n\times Z$.
We first computed the image of $C_n\times Z$ under the map given by the bilinear forms $B_j(p,q)$,
it is the quadric in $\PP^5$ defined by:
$$
Q_{n}\,:\quad q_{N}\,=\,0,\qquad\mbox{with}\quad
q_{n}\,:=\, x_1x_4 \,-\, x_2x_5~.
$$
Thus we have an interpretation of the quadric that cuts out $F=\phi_L(E)$ 
in terms of the geometry of $S_{10}$. 
The trivial conic bundle
$n\times S\subset S^2$ maps first of all to a subvariety of $S^{[2]}$ that is birationally
isomorphic to the trivial bundle $\PP^1$-bundle over $S$, and the image of this
subvariety under $\phi_L$ is $F=Q_n\cap Y$.

\subsubsection{The singular surface of $Y$} 
The singular surface of $Y$ is rather remarkable in that its singular locus 
consists of $60$ points, which we call the very singular points of $Y$.
These points are rational over the splitting field $L$ of the polynomial $x^4-2$.
Notice that the roots of this polynomial are $\alpha_k:=i^k\sqrt[4]2$, $k=0,1,2,3$, and thus
$i=\alpha_1/\alpha_0\in L$, with $i^2=-1$, and also $\zeta_8:=(1+i)/\sqrt{2}\in L$, with 
$\zeta_8^2=i$. Thus $L$ contains the field $\QQ(\zeta_8)$
of eight-roots of unity. 
Only two of the singular points are defined over $\QQ$, there are
$11$ other Galois orbits, one of length $2$, six of length $4$ and four of length $8$.

\subsubsection{The very singular points}
The lines on the Del Pezzo quintic threefold $Z$ in $\PP^6$ 
are parametrized by $\PP^2$. The general line can be found as the image of a secant line
to the rational curve $n'\subset\PP^3$ under the map $\phi$.
Since $Z$ is defined by the last 5 quadrics 
of the six in the list given in Section \ref{epw} that define the K3 surface $S\subset\PP^6$, 
the first quadric will intersect a line on $Z$, not contained in $S_{10}$, 
in two points. This gives a rational map 
$\PP^2\rightarrow S^{[2]}$.
Under the map $S^{[2]}\rightarrow Y$ the image of this $\PP^2$
is contracted to $(1:0:\ldots:0)\in Y$, which is in fact one of the sixty 
very singular points of $Y$. 
See \cite[section 3.2]{O'G-Dual}  and \cite{Beri} for the map 
$S^{[2]}\rightarrow Y$. 

Let $C\subset S_{10}$ be a smooth conic and let $\PP_C\subset\PP^6$ be the plane it spans. 
Any quadric
containing $S_{10}$ either contains $\PP_C$ or intersects $\PP_C$ in $C$. 
If the quadric $F=0$ contains
$\PP_C$, then for any two points $p,q\in C$ one has $B(p,q)=0$ 
where $B$ is the associated bilinear form. 
Thus $C^{[2]}\cong\PP^2\subset S^{[2]}$ is contracted under $\phi_L:S^{[2]}\rightarrow Y$ 
and the image is a very singular point of $Y$
(cf.\ \cite[Section 4]{O'G-Double} but notice that $S_{10}$ does contain lines, 
in fact there are $20$ lines on $S_{10}$, 
for example $(is:it:\zeta_8t:s:\zeta_8^3t:0:t)$ parametrizes a line on $S_{10}$).
Similarly, in case the conic $C=l\cup m$ is reducible, 
the image of $l\times m\subset S^2$ in $S^{[2]}$
is contracted by $\phi_L$.

In particular the plane $y_0=\ldots=y_3=0$ spanned by the conic $C_n\subset S_{10}$
lies in the quadric $F_i=0$ except if $i=2$ and thus $\sym^2(C_n)$ is contracted to
the very singular point $(0:0:0:1:0:0)$.
The other conics that are contracted are not defined over the rationals
(there are $59$ conics on $S_{10}$, $34$ of them are reducible).

\

\

\end{document}